\documentclass[a4paper,11pt]{article}
\usepackage[latin1]{inputenc}
\usepackage[english]{babel}
\usepackage{amsmath}
\usepackage{amsfonts}
\usepackage{amssymb}
\usepackage{epsfig}
\usepackage{amsopn}
\usepackage{amsthm}
\usepackage{color}
\usepackage{graphicx}
\usepackage{subfigure}
\usepackage{enumerate}
\newtheorem{theorem}{Theorem}[section]

\newtheorem{lemma}[theorem]{Lemma}

\DeclareMathOperator\arctanh{arctanh}
\newtheorem*{theorem*}{Theorem}
\newtheorem*{lemma*}{Lemma}
\newtheorem*{remark*}{Remark}
\newtheorem*{definition*}{Definition}
\newtheorem*{proposition*}{Proposition}
\newtheorem*{corollary*}{Corollary}
\numberwithin{equation}{section}
\newcommand{\real}{\mathbb{R}}

\let\ced=\c         

\def\qed{\,\unskip\kern 6pt \penalty 500
\raise -2pt\hbox{\vrule \vbox to8pt{\hrule width 6pt
\vfill\hrule}\vrule}\par}
\definecolor{darkblue}{rgb}{0.05, .05, .65}
\definecolor{darkgreen}{rgb}{0.1, .65, .1}
\definecolor{darkred}{rgb}{0.8,0,0}
\newcommand{\beqn}{\begin{equation}}
\newcommand{\eeqn}{\end{equation}}
\newcommand{\bear}{\begin{eqnarray}}
\newcommand{\eear}{\end{eqnarray}}
\newcommand{\bean}{\begin{eqnarray*}}
\newcommand{\eean}{\end{eqnarray*}}
\begin{document}

\title{\huge \bf Extinction and non-extinction profiles for the sub-critical fast diffusion equation with weighted source}

\author{
\Large Razvan Gabriel Iagar\,\footnote{Departamento de Matem\'{a}tica
Aplicada, Ciencia e Ingenieria de los Materiales y Tecnologia
Electr\'onica, Universidad Rey Juan Carlos, M\'{o}stoles,
28933, Madrid, Spain, \textit{e-mail:} razvan.iagar@urjc.es},\\
[4pt] \Large Ana Isabel Mu\~{n}oz\,\footnote{Departamento de Matem\'{a}tica
Aplicada, Ciencia e Ingenieria de los Materiales y Tecnologia
Electr\'onica, Universidad Rey Juan Carlos, M\'{o}stoles,
28933, Madrid, Spain, \textit{e-mail:} anaisabel.munoz@urjc.es},
\\[4pt] \Large Ariel S\'{a}nchez,\footnote{Departamento de Matem\'{a}tica
Aplicada, Ciencia e Ingenieria de los Materiales y Tecnologia
Electr\'onica, Universidad Rey Juan Carlos, M\'{o}stoles,
28933, Madrid, Spain, \textit{e-mail:} ariel.sanchez@urjc.es}\\
[4pt] }
\date{}
\maketitle

\begin{abstract}
We establish both extinction and non-extinction self-similar profiles for the following fast diffusion equation with a weighted source term 
$$
\partial_tu=\Delta u^m+|x|^{\sigma}u^p,
$$
posed for $(x,t)\in\real^N\times(0,\infty)$, $N\geq3$, in the sub-critical range of the fast diffusion equation $0<m<m_c=(N-2)/N$. We consider $\sigma>0$ and $\max\{p_c(\sigma),1\}<p<p_L(\sigma)$, where 
$$
p_c(\sigma)=\frac{m(N+\sigma)}{N-2}, \qquad p_L(\sigma)=1+\frac{\sigma(1-m)}{2}.
$$
We show that, on the one hand, positive self-similar solutions at any time $t>0$, in the form 
$$
u(x,t)=t^{\alpha}f(|x|t^{\beta}), \qquad f(\xi)\sim C\xi^{-(N-2)/m}, \qquad \alpha>0, \ \beta>0
$$
exist, provided $0<m<m_s=(N-2)/(N+2)$ and $p_s(\sigma)=m(N+2\sigma+2)/(N-2)<p<p_L(\sigma)$. On the other hand, we prove that there exists $p_0(\sigma)\in(p_c(\sigma),p_s(\sigma))$ such that self-similar solutions presenting finite time extinction are established both for $p\in(p_0(\sigma),p_s(\sigma))$ and for $p\in(p_s(\sigma),p_L(\sigma))$, but with profiles $f(\xi)$ having different spatially decreasing tails as $|x|\to\infty$. We also prove non-existence of self-similar solutions in complementary ranges of exponents to the ones described above or if $m\geq m_c$. 
\end{abstract}

\

\noindent {\bf Mathematics Subject Classification 2020:} 35B33, 35B36, 35C06, 35K57, 35K59.

\smallskip

\noindent {\bf Keywords and phrases:} fast diffusion equations, sub-critical range, finite time extinction, weighted reaction, Sobolev critical exponent, self-similar solutions.

\section{Introduction}

The fast diffusion equation 
\begin{equation}\label{FDE}
u_t=\Delta u^m, \qquad 0<m<1, \qquad (x,t)\in\real^N\times(0,\infty)
\end{equation}
has been established as a classical model in the theory of the nonlinear diffusion and proved to be of a large interest for mathematicians due to its variety of mathematical features depending on various ranges of exponents $m$. We refer the interested reader to the excellent monograph \cite{VazSmooth}, which gives an overview of the most important advances that have been achieved in its analysis. In particular, if the spatial dimension is $N\geq3$, the interval $m\in(0,1)$ is split into three sub-ranges by the following relevant exponents
\begin{equation}\label{crit.m}
m_c:=\frac{N-2}{N}, \qquad m_s:=\frac{N-2}{N+2},
\end{equation}
known as \emph{the critical exponent}, respectively \emph{the Sobolev exponent} of the fast diffusion equation. The critical exponent $m_c$ separates two very different ranges. On the one hand, in the super-critical range $m\in(m_c,1)$ Eq. \eqref{FDE} maintains its property of mass conservation inherited from the porous medium equation and solutions with integrable initial conditions remain positive at any time $t>0$. On the other hand, in the sub-critical range $0<m<m_c$ mass of the solutions is lost through infinity and any solution decaying to zero sufficiently fast as $|x|\to\infty$ presents \emph{finite time extinction}, that is, there exists $T\in(0,\infty)$ such that $u(x,t)>0$ for any $x\in\real^N$ and $t\in(0,T)$ but $u(x,T)=0$ for any $x\in\real^N$. In the latter range, King \cite{Ki93} followed by the rigorous analysis by Peletier and Zhang \cite{PZ95} established the existence of a unique pair of exponents $(\alpha(m),\beta(m))$ and of a branch of self-similar solutions in backward form satisfying
\begin{equation}\label{backward}
u(x,t)=(T-t)^{\alpha(m)}f(|x|(T-t)^{\beta(m)}), \qquad T>0, \qquad \alpha(m)(m-1)+2\beta(m)=-1,
\end{equation}
that were called \emph{anomalous self-similar solutions}, since the exponents $\alpha(m)$ and $\beta(m)$ cannot be established by simple algebraic calculations, but as the outcome of an analysis employing dynamical systems techniques (see also \cite[Section 7.2]{VazSmooth} for a survey of this theory). In particular, the exponent $m_s$ in \eqref{crit.m} is the only one allowing for an explicit pair of exponents and profile, and it is strongly related to the Yamabe flow in Riemannian geometry, see \cite[Section 7.5]{VazSmooth} for more details. The anomalous self-similar solutions proved to be extremely useful for the description of the dynamics of the solutions to Eq. \eqref{FDE}, being the patterns to which general radially symmetric solutions converge as shown by Galaktionov and Peletier \cite{GP97}. The condition of radial symmetry has been removed from the previous large time behavior result only for $m=m_s$ by Del Pino and Saez \cite{dPS01}. More properties of the fast diffusion equation related to both critical exponents $m_c$ and $m_s$ can be found in \cite{VazSmooth}. 

In the present paper, our aim is to explore the form and properties of self-similar solutions to the fast diffusion equation with a spatially inhomogeneous reaction
\begin{equation}\label{eq1}
u_t=\Delta u^m+|x|^{\sigma}u^p, \qquad (x,t)\in\real^N\times(0,\infty),
\end{equation} 
with 
\begin{equation}\label{pLsigma}
0<m<1, \qquad \sigma>0, \qquad 1<p<p_L(\sigma):=1+\frac{\sigma(1-m)}{2}.
\end{equation}
Eq. \eqref{eq1} is a competitive equation, its most interesting feature being the mixing between the fast diffusion term, who is already very rich in interesting properties as explained in the previous paragraph, and a spatially inhomogeneous source term introducing mass into the system. This competition is even more striking in the sub-critical range $0<m<m_c$ where, as explained above, on the one hand the fast diffusion involves a loss of mass of the solutions through the infinity (a description of this fact can be read in \cite[Section 5.5]{VazSmooth}), while on the other hand the source term introduces mass into the model. Moreover, a second competition, the one between the influence of neighborhoods of the origin (where the reaction term is at least formally very small) and of regions at positive (and big) distance from the origin (where the reaction term becomes very large) on the evolution of a solution, is hidden in the form of Eq. \eqref{eq1}. We thus expect to have an even richer bunch of mathematical phenomena related to the properties of solutions to Eq. \eqref{eq1}.

Eq. \eqref{eq1} has been investigated rather deeply in the semilinear case $m=1$ and the slow diffusion case $m>1$, specially when the source term is spatially homogeneous, that is, with $\sigma=0$. In the previous range of $m$, the main feature of this equation is the finite time blow-up of its solutions, and nowadays many properties of solutions to Eq. \eqref{eq1} with $\sigma=0$ are known, including when finite time blow-up takes place, blow-up rates and profiles (see the monographs \cite{QS} for $m=1$ and \cite{S4} for $m>1$). A relevant critical exponent known as \emph{the Fujita exponent}
\begin{equation}\label{Fujita}
p_F=m+\frac{2}{N}
\end{equation}
splits between the range where all non-trivial solutions blow up in finite time, that is, $1<p<p_F$, and the range where there exist global solutions (in time) $p>p_F$. The semilinear case of Eq. \eqref{eq1} with $\sigma>0$ had been considered in a number of already classical papers such as \cite{BK87, BL89, Pi97, Pi98}, also in connection with the finite time blow-up, while Filippas and Tertikas \cite{FT00} performed an analysis of self-similar solutions to Eq. \eqref{eq1} with $m=1$. More recently Mukai and Seki \cite{MS21} managed to describe, also for $m=1$ and $\sigma\in(-2,\infty)$ but $p$ sufficiently large, the rather unusual phenomenon of blow-up of Type II, that is, with variable rates and geometric patterns. Eq. \eqref{eq1} with $m>1$ has been analyzed by Qi \cite{Qi98} and Suzuki \cite{Su02}, who established several critical exponents including in particular the Fujita-type exponent $p_F(\sigma)=m+(\sigma+2)/N$, and in the former, a self-similar solution which is global in time has been constructed. Notice that \cite{Qi98} also extends its results to the super-critical fast diffusion. The authors and their collaborators started in the last years a program of studying Eq. \eqref{eq1} and in the range $m>1$, a number of results concerning the analytical properties of self-similar solutions have been obtained, see for example \cite{IS19, IS22, IS22b, IMS22, IMS23, ILS23, IL22} and references therein. It has been shown in these works that the sign of the expression 
\begin{equation}\label{const.L}
L:=\sigma(m-1)+2(p-1)
\end{equation}
has a decisive effect on the dynamics of Eq. \eqref{eq1}, which leads us to consider the exponent $p_L(\sigma)$ defined in \eqref{pLsigma}: observe that $p<p_L(\sigma)$ is equivalent to $L<0$.

Let us give some precedents related to our main object of interest, that is, Eq. \eqref{eq1} with fast diffusion $m\in(0,1)$. Still with $\sigma=0$, it has been shown in \cite{Qi93, MM95, GuoGuo01} that in the super-critical range $m_c<m<1$, the exponent $p_F>1$ given in \eqref{Fujita} still plays the role of a Fujita-type exponent in the previously described sense of limiting the existence and non-existence of global solutions. Later Maing\'e \cite{Ma08} extended the results established in \cite{GuoGuo01} and described the connection between the decay rate of an initial condition $u_0$ as $|x|\to\infty$ and the time frame of existence of the solution to the Cauchy problem with data $u_0$ for the whole fast diffusion range $m\in(0,1)$. However, let us remark that none of these works entered the sub-critical fast diffusion range and all them (and also \cite{Qi98}) have as cornerstone the phenomenon of finite time blow-up. The fast diffusion equation with weighted reaction and localized weight has been considered in \cite{BZZ11}.

A strongly related recent work is \cite{IS22c}, where two of the authors analyze Eq. \eqref{eq1} for the critical exponent $p=p_L(\sigma)$ which is expected to be a borderline case between different behaviors. In this paper, the authors enter the sub-critical fast diffusion range and show that for $m\in(0,m_c)$ there exists a unique branch of \emph{eternal} anomalous self-similar solutions in exponential form
\begin{equation}\label{exp.ss}
u(x,t)=e^{\alpha(m) t}f(|x|e^{-\beta(m) t}), \qquad \alpha(m)=-\frac{2}{1-m}\beta(m), 
\end{equation}
such that $\alpha(m)=\beta(m)=0$ for $m=m_s$, $\alpha(m)>0$ and $\beta(m)<0$ for $m\in(0,m_s)$, respectively $\alpha(m)<0$ and $\beta(m)>0$ for $m\in(m_s,m_c)$. We thus notice that both critical exponents in \eqref{crit.m} come into play in this analysis. Taking into account also the great importance that the branch of anomalous solutions has in the theory of fast diffusion, it comes as a natural question to study whether there are self-similar solutions of any type when $p\neq p_L(\sigma)$.

\medskip

\noindent \textbf{Main results}. Our goal in this paper is to classify the self-similar solutions allowed by Eq. \eqref{eq1} in the range of exponents $0<m<1$, $\sigma>0$ and $1<p<p_L(\sigma)$. As we shall see, despite the fact that there exists a source term in the equation (whose expected effect would be the finite time blow-up of at least some solutions, as explained above), in reality we will obtain self-similar solutions in only two forms that are strikingly different from blow-up. In the forthcoming study, the following two critical exponents of reaction will be very important:
\begin{equation}\label{crit.p}
p_c(\sigma):=\frac{m(N+\sigma)}{N-2}, \qquad p_s(\sigma):=\frac{m(N+2\sigma+2)}{N-2},
\end{equation}
provided $N\geq3$, both exponents being set to $+\infty$ by convention in lower dimensions $N\in\{1,2\}$. These exponents have been identified in \cite{IS23} through some transformations, but they are natural extensions of the ones given in \cite{FT00} for $m=1$ and any $\sigma>-2$ or of the ones in \cite{GV97} for $m>1$ but $\sigma=0$. We also fix throughout the paper the self-similar exponents 
\begin{equation}\label{ss.exponents}
\alpha=-\frac{\sigma+2}{L}, \qquad \beta=-\frac{p-m}{L}
\end{equation}
and notice that $\alpha>0$, $\beta>0$ since $p<p_L(\sigma)$. 

\medskip 

\noindent \textbf{A. Global in time self-similar solutions.} These are self-similar solutions in the following form 
\begin{equation}\label{forward}
u(x,t)=t^{\alpha}f(|x|t^{\beta}). 
\end{equation}
Plugging the ansatz \eqref{forward} into Eq. \eqref{eq1}, we deduce that $\alpha$, $\beta$ are the ones defined in \eqref{ss.exponents} and the profile $f(\xi)$, $\xi=|x|t^{\beta}$, solves the differential equation
\begin{equation}\label{ODE.forward}
(f^m)''(\xi)+\frac{N-1}{\xi}(f^m)'(\xi)-\alpha f(\xi)-\beta\xi f'(\xi)+\xi^{\sigma}f^p(\xi)=0.
\end{equation}
We will be then looking for profiles $f(\xi)$ solving \eqref{ODE.forward} together with the initial conditions $f(0)=A>0$, $f'(0)=0$ in order to cope with the condition of radial symmetry imposed to the self-similar solutions. We prove the following existence and non-existence result:
\begin{theorem}\label{th.forward}
Let $m$, $p$ and $\sigma$ satisfy the conditions \eqref{pLsigma}. 
\begin{enumerate}
  \item Let $N\geq3$, $m\in(0,m_s)$ and $\max\{1,p_s(\sigma)\}<p<p_L(\sigma)$. Then there exist self-similar solutions to Eq. \eqref{eq1} in the form \eqref{forward} such that their profiles $f(\xi)$ satisfy the following local behavior near the origin
\begin{equation}\label{beh.P0f}
f(\xi)\sim\left[D-\frac{\alpha(1-m)}{2mN}\xi^2\right]^{-1/(1-m)}, \qquad {\rm as} \ \xi\to0
\end{equation}
and the fast decay rate at infinity 
\begin{equation}\label{beh.P1}
f(\xi)\sim C\xi^{-(N-2)/m}, \qquad {\rm as} \ \xi\to\infty,
\end{equation}
where $C$, $D>0$ are positive constants.
  \item In the same conditions as in Part 1, there exist also self-similar solutions to Eq. \eqref{eq1} in the form \eqref{forward} with the same local behavior \eqref{beh.P0f} as $\xi\to0$ but the slow decay rate at infinity
\begin{equation}\label{beh.P2}
f(\xi)\sim C\xi^{-(\sigma+2)/(p-m)}, \qquad {\rm as} \ \xi\to\infty.
\end{equation}
where $C>0$ is a positive constant. 
\item There is no self-similar solution to Eq. \eqref{eq1} in the form \eqref{forward} if at least one of the following conditions $m\in[m_s,1)$, $N\in\{1,2\}$, or $1<p< p_s(\sigma)$ is fulfilled.
\end{enumerate}
\end{theorem}
Notice that the decay rate \eqref{beh.P1} at infinity \emph{matches with the one of the anomalous solutions} for the fast diffusion equation \eqref{FDE}. We can thus think about these solutions as analogous ones to the branch of anomalous solutions, although in our case the exponents $\alpha$ and $\beta$ are explicit. Let us emphasize here on the importance of the solutions with the fast decay rate in nonlinear diffusion problems, as it is already a well known fact that they are expected to become asymptotic patterns for a large class of general solutions, as seen in \cite{GP97} for the anomalous branch of the sub-critical fast diffusion equation but also for different nonlinear diffusion equations, see for example \cite{IL14, BIS16}. Let us remark here too, the importance for the classification of the critical Sobolev exponents, both $m_s$ and $p_s(\sigma)$ coming strongly into play in Theorem \ref{th.forward}.

\medskip 

\noindent \textbf{B. Self-similar solutions with finite time extinction.} These are self-similar solutions of the form 
\begin{equation}\label{ext.ss}
u(x,t)=(T-t)^{\alpha}f(|x|(T-t)^{\beta}).
\end{equation}
Plugging the ansatz \eqref{ext.ss} into Eq. \eqref{eq1}, we find the same values of $\alpha$ and $\beta$ as in \eqref{ss.exponents} and the profiles $f(\xi)$, $\xi=|x|(T-t)^{\beta}$, satisfy the following differential equation
\begin{equation}\label{ODE.extinction}
(f^m)''(\xi)+\frac{N-1}{\xi}(f^m)'(\xi)+\alpha f(\xi)+\beta\xi f'(\xi)+\xi^{\sigma}f^p(\xi)=0.
\end{equation}
Since $\alpha$, $\beta>0$ and $f(0)\in(0,\infty)$, we readily observe that solutions given by \eqref{ext.ss} vanish at time $t=T$ with extinction rate $(T-t)^{\alpha}$. For these solutions we have the following existence and non-existence result.
\begin{theorem}\label{th.extinction}
Let $m$, $p$ and $\sigma$ satisfy the conditions \eqref{pLsigma}. 
\begin{enumerate}
  \item Let $N\geq3$ and $m\in((N-2)/(N+2+2\sigma),m_s)$. Then there exists $p_0(\sigma)\in(p_c(\sigma),p_s(\sigma))$, $p_0(\sigma)>1$ such that for any $p\in(p_0(\sigma),p_s(\sigma))$, there exist self-similar solutions in the form \eqref{ext.ss} to Eq. \eqref{eq1} such that their profiles $f(\xi)$ are decreasing and satisfy the following local behavior near the origin
\begin{equation}\label{beh.P0b}
f(\xi)\sim\left[D+\frac{\alpha(1-m)}{2mN}\xi^2\right]^{-1/(1-m)}, \qquad {\rm as} \ \xi\to0
\end{equation}
and the fast decay rate at infinity \eqref{beh.P1}, where $D>0$ is a positive constant. Moreover, there are no self-similar solutions with local behavior \eqref{beh.P0b} and the slow decay rate \eqref{beh.P2} as $\xi\to\infty$.
  \item Let $N\geq3$ and $m\in(0,m_s)$. Then, for any $p\in(\max\{p_s(\sigma),1\},p_L(\sigma))$, there exist self-similar solutions in the form \eqref{ext.ss} to Eq. \eqref{eq1} such that their profiles $f(\xi)$ are decreasing and satisfy the local behavior \eqref{beh.P0b} as $\xi\to0$ and the slow decay rate \eqref{beh.P2} as $\xi\to\infty$ and there are no profiles $f(\xi)$ with the fast decay rate \eqref{beh.P1} as $\xi\to\infty$.
  \item If at least one of the following conditions $N\in\{1,2\}$, $m\geq m_c$ is fulfilled, then there are no self-similar solutions to Eq. \eqref{eq1} in the form \eqref{ext.ss}. Moreover, there exists $p_1(\sigma)\in(p_c(\sigma),p_0(\sigma)]$ such that, for any $p\in(1,p_1(\sigma))$, there are no self-similar solutions to Eq. \eqref{eq1} with finite time extinction.
\end{enumerate}
\end{theorem}
Observe that, if $m\in(0,(N-2)/(N+2+2\sigma)]$, we have $p_s(\sigma)\leq1$ and thus the second item in Theorem \ref{th.extinction} applies.

The most interesting outcome of this paper is, in our opinion, the \textbf{co-existence} in some ranges of exponents of global in time self-similar solutions and solutions presenting finite time extinction. As explained in the previous paragraphs, almost all the research related to equations with source terms goes around the finite time blow-up of solutions, which is a natural expectation induced by the reaction term. However, we show that if $0<m<m_s$ and $1<p<p_L(\sigma)$, Eq. \eqref{eq1} admits solutions whose dynamics is completely opposite to blow-up, which is the result of the big loss of mass that the diffusion term involves in this range (and that in some cases, cannot be compensated by the effect of the source term).

\medskip

\noindent \textbf{C. A family of stationary solutions for} $\mathbf{p=p_s(\sigma)}$. As one can see by inspecting the proofs of Theorems \ref{th.forward} and \ref{th.extinction}, there are no self-similar solutions in the previous forms and with the fast decay \eqref{beh.P1} if $p=p_s(\sigma)$, although this critical exponent limits the ranges of existence for them. Instead, if $p=p_s(\sigma)$ there exists a one-parameter family of explicit \emph{stationary} solutions, more precisely, for any $C>0$ we have
\begin{equation}\label{sol.sobolev}
U_C(x)=\left[\frac{(N-2)(N+\sigma)C}{(|x|^{\sigma+2}+C)^2}\right]^{(N-2)/2m(\sigma+2)}.
\end{equation}
Notice that $U_C(0)\in(0,\infty)$, $U_C'(0)=0$ and $U_C$ has the fast decay rate \eqref{beh.P1} as $|x|\to\infty$. We can also observe that the stationary solutions given by \eqref{sol.sobolev} are not limited to our range of exponents: they exist for any $\sigma>-2$, $m>0$ whenever $p=p_s(\sigma)$. A particular case of them has been obtained in \cite[Section 3.5]{IS22c} for $p=p_L(\sigma)$ and $m=m_s$. We can interpret these stationary solutions with the fast decay rate as $|x|\to\infty$ as the limiting behavior between solutions with a decreasing $L^{\infty}$ norm and extinction (as given by Theorem \ref{th.extinction} for $p\in(p_0(\sigma),p_s(\sigma))$) and global solutions whose $L^{\infty}$ norm increases with time (as given by Theorem \ref{th.forward}).

\medskip

\noindent \textbf{Remarks. 1.} We have the following equivalences related to the critical exponents 
\begin{equation}\label{interm1}
\begin{split}
&p_c(\sigma)<p_L(\sigma) \qquad {\rm iff} \qquad m\in(0,m_c), \\
&p_s(\sigma)<p_L(\sigma) \qquad {\rm iff} \qquad m\in(0,m_s),
\end{split}
\end{equation}
justifying some of the limitations in the statements of Theorems \ref{th.forward} and \ref{th.extinction}. 

\textbf{2}. Notice that the profiles with the local behavior \eqref{beh.P1} as $\xi\to\infty$ give rise to solutions in $L^1(\real^N)$ in the range $m\in(0,m_c)$, since $(N-2)/m>N$. The self-similar solutions whose profiles satisfy the local behavior \eqref{beh.P2} as $\xi\to\infty$ belong to $L^1(\real^N)$ only if $(\sigma+2)/(p-m)>N$, which is fulfilled for 
$$
p<p_F(\sigma)=m+\frac{\sigma+2}{N},
$$
which is known as the Fujita-type exponent. Since profiles with local behavior \eqref{beh.P2} only exist for 
$$
p>\max\{1,p_s(\sigma)\},
$$ 
we infer by simple calculations that this situation occurs if 
$$
\sigma<N-2, \qquad \frac{N-2-\sigma}{N}<m<\frac{N-2}{2N}<m_s.
$$
In this case, we notice that the exponent $p_F(\sigma)$ does not play the usual role of the Fujita exponent, despite the fact that $p_F(\sigma)>1$, since there are solutions which do not blow up in finite time for $\max\{1,p_s(\sigma)\}<p<p_F(\sigma)$.

\textbf{3}. Let us observe that the self-similar profiles with local behavior \eqref{beh.P0f} as $\xi\to0$ are increasing in a right neighborhood of the origin and will have a point of maximum at some $\xi_0\in(0,\infty)$. The maximum of the corresponding self-similar solution at time $t>0$ is achieved at $|x|=\xi_0t^{-\beta}$, moving towards the origin with time, and we find that these solutions have an increasing $L^{\infty}$ norm, more precisely
$$
\|u(t)\|_{\infty}=t^{\alpha}f(\xi_0)=t^{\alpha}\max\{f(\xi):\xi\geq0\}.
$$
Contrasting with this fact, fixing $x\in\real^N$ with $|x|>0$, the evolution of the global self-similar solutions as $t\to\infty$ and at this fixed $x$ is given by 
$$
u(x,t)\sim t^{\alpha}(|x|t^{\beta})^{-(N-2)/m}=t^{-(N-2)(p_c(\sigma)-p)/Lm}|x|^{-(N-2)/m}, 
$$
decreasing with time as $p>p_c(\sigma)$ and $L<0$, if the profile $f(\xi)$ of the self-similar solution presents the fast decay \eqref{beh.P1} as $\xi\to\infty$, respectively 
$$
u(x,t)\sim t^{\alpha}(|x|t^{\beta})^{-(\sigma+2)/(p-m)}=|x|^{-(\sigma+2)/(p-m)}, \qquad {\rm as} \ t\to\infty
$$
which does not depend on time, if the profile $f(\xi)$ of the self-similar solution presents the slow decay \eqref{beh.P2} as $\xi\to\infty$. We thus have found self-similar solutions whose $L^{\infty}$ norm grows up as $t\to\infty$ but it is attained at points translating towards the origin, while at fixed positive points the solutions tend either to a stationary profile or decay to zero with time.

\textbf{4}. Contrary to the previous remark, self-similar profiles with local behavior \eqref{beh.P0b} as $\xi\to0$ are decreasing and have a maximum at $\xi=0$. Hence, the corresponding self-similar solutions will have at any time $t>0$ their peak at the origin, where at least formally the reaction is very weak. This fact suggests the domination of the diffusion effect in their behavior, thus leading to finite time extinction. In particular, we have
$$
\|u(t)\|_{\infty}=u(0,t)=(T-t)^{\alpha}f(0), \qquad t\in(0,T),
$$
which also implies that the extinction rate is $(T-t)^{\alpha}$ for these solutions, independently of their tail as $\xi\to\infty$.

We end this presentation by plotting in Figure \ref{fig1} two self-similar solutions, one global and one with finite time extinction, taken at different times, showing thus the evolution of the $L^{\infty}$ norm, point of maximum and fixed points as explained in the previous remarks.

\begin{figure}[ht!]
  \begin{center}
  \subfigure[Global self-similar solutions]{\includegraphics[width=7.5cm,height=6cm]{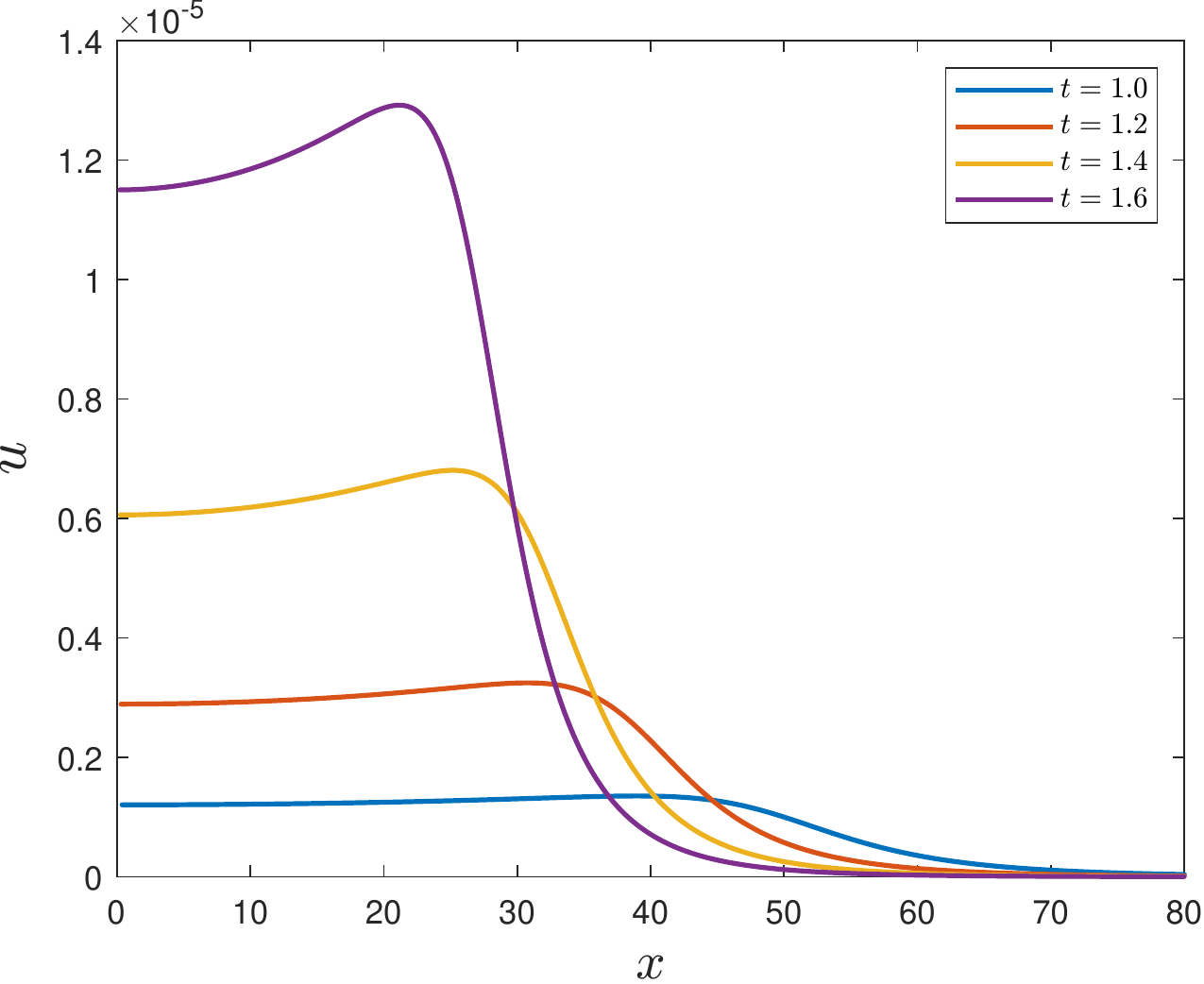}}
  \subfigure[Self-similar solutions with extinction]{\includegraphics[width=7.5cm,height=6cm]{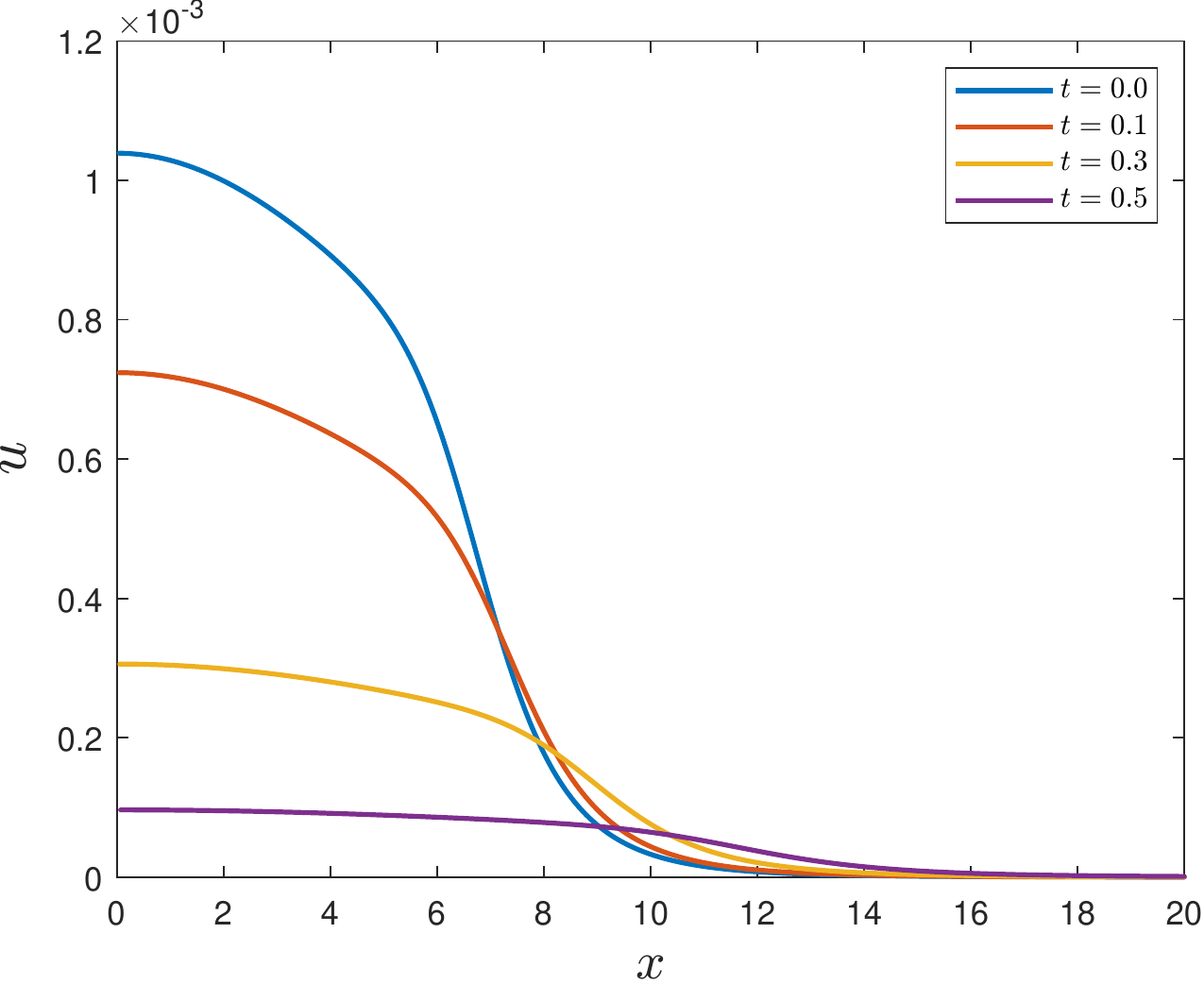}}
  \end{center}
  \caption{A global self-similar solution and a self-similar solution with finite time extinction. Experiments for $m=0.25$, $N=4$, $\sigma=10$ and $p=3.5$, respectively $p=3$ and $T=1$.}\label{fig1}
\end{figure}

\medskip 

\noindent \textbf{Organization of the paper}. The main tool when classifying self-similar solutions to a nonlinear diffusion equation is a phase plane analysis of some three-dimensional, quadratic dynamical systems into which the differential equations \eqref{ODE.forward} and \eqref{ODE.extinction} are mapped through some inspired changes of variable. This system is deduced in Section \ref{sec.local} for the equation \eqref{ODE.forward} and the local analysis of its finite critical points is performed in the same section, being then completed with the analysis of the critical points at infinity in Sections \ref{sec.inf} and \ref{sec.top}. The main point in the global analysis is a bifurcation that occurs at $p=p_s(\sigma)$ and which is explained in Section \ref{sec.invariant} at the level of an invariant plane in which all the significant critical points lie. The global analysis of the first system, establishing the connections in the phase space which represent self-similar profiles in the $\xi$ variable, is performed in Section \ref{sec.global} and includes the proof of Theorems \ref{th.forward}. We adapt the dynamical system and the local analysis of its critical points to cope with the second differential equation \eqref{ODE.extinction} in Section \ref{sec.ext.local}, and the proof of Theorem \ref{th.extinction} is the subject of Section \ref{sec.extinction}. The last, shorter Section \ref{sec.sobolev} describes how we found the stationary solutions \eqref{sol.sobolev} in the case $p=p_s(\sigma)$.

\section{The dynamical system. Local analysis}\label{sec.local}

We focus first on the self-similar solutions that are global in time and whose profiles $f(\xi)$, $\xi=|x|t^{\beta}$, solve the differential equation \eqref{ODE.forward}. In order to study this non-autonomous equation, we convert it into a quadratic dynamical system through the following change of variable: 
\begin{equation}\label{PSchange}
X(\eta)=\frac{\alpha}{m}\xi^2f^{1-m}(\xi), \qquad Y(\eta)=\frac{\xi f'(\xi)}{f(\xi)}, \qquad Z(\eta)=\frac{1}{m}\xi^{\sigma+2}f^{p-m}(\xi),
\end{equation}
with the new independent variable $\eta=\ln\,\xi$. Eq. \eqref{ODE.forward} is transformed, after direct calculations, into the system 
\begin{equation}\label{PSsyst}
\left\{\begin{array}{ll}\dot{X}=X(2+(1-m)Y), \\ \dot{Y}=X-(N-2)Y-Z-mY^2+\frac{p-m}{\sigma+2}XY, \\ \dot{Z}=Z(\sigma+2+(p-m)Y).\end{array}\right.
\end{equation}
Since we are looking only to non-negative self-similar solutions, we infer that $X\geq0$, $Z\geq0$ and the planes $\{X=0\}$ and $\{Z=0\}$ are invariant for the system \eqref{PSsyst}, while $Y$ might change sign in dependence on the monotonicity of the profiles. We will thus analyze the trajectories of the system \eqref{PSsyst} and select the ones that are interesting in terms of profiles when undoing the change of variable \eqref{PSchange}. The system \eqref{PSsyst} has four finite critical points
\begin{equation*}
\begin{split}
P_0=(0,0,0), \ \ &P_1=\left(0,-\frac{N-2}{m},0\right), \ \ P_2=\left(0,-\frac{\sigma+2}{p-m},\frac{(N-2)(\sigma+2)(p-p_c(\sigma))}{(p-m)^2}\right),\\
&P_3=\left(-\frac{2(\sigma+2)(mN-N+2)}{L(1-m)},-\frac{2}{1-m},0\right),
\end{split}
\end{equation*}
noticing that $P_2$ exists only for $p>p_c(\sigma)$, where we recall that $p_c(\sigma)$ is defined in \eqref{crit.p}, and $P_3$ exists only for $m\in[m_c,1)$ since we have $X\geq0$. We analyze below the local dynamics of the system \eqref{PSsyst} in a neighborhood of these points, assuming for the moment that $N\geq3$. We skip by now the local analysis near the critical point $P_3$, which is very similar to the one given in Lemma \ref{lem.P3} to which we refer the reader.
\begin{lemma}\label{lem.P0}
In a neighborhood of the critical point $P_0$, the system \eqref{PSsyst} has a two-dimensional unstable manifold and a one-dimensional stable manifold. The profiles contained in orbits going out of $P_0$ satisfy the local behavior \eqref{beh.P0f}. 
\end{lemma}
\begin{proof}
The linearization of the system \eqref{PSsyst} in a neighborhood of $P_0$ has the matrix 
$$
M(P_0)=\left(
         \begin{array}{ccc}
           2 & 0 & 0 \\
           1 & -(N-2) & -1 \\
           0 & 0 & \sigma+2 \\
         \end{array}
       \right),
$$
with two positive eigenvalues $\lambda_1=2$ and $\lambda_3=\sigma+2$ and one negative eigenvalue $\lambda_2=-(N-2)$, since we are working under the assumption that $N\geq3$. Notice first that the one-dimensional stable manifold is completely contained in the $Y$-axis (which is an invariant line), since the second eigenvector is $e_2=(0,1,0)$. Let us now move to the two-dimensional unstable manifold. In a first order approximation, we readily get from the first and third equation of the system \eqref{PSsyst} that 
\begin{equation}\label{interm2}
Z(\eta)\sim CX^{(\sigma+2)/2}(\eta), \qquad {\rm as} \ \eta\to-\infty.
\end{equation}
We next observe that \eqref{interm2} implies that, in a neighborhood of $P_0$, the $Z$ coordinate is negligible with respect to the $X$ coordinate, thus 
$$
\frac{dY}{dX}\sim\frac{-(N-2)Y+X}{2X},
$$ 
which gives by integration that 
\begin{equation}\label{interm3}
Y(\eta)\sim\frac{X(\eta)}{N}+C_1X(\eta)^{-(N-2)/2}, \qquad {\rm as} \ \eta\to-\infty.
\end{equation}
Since the trajectories are assumed to pass by $P_0$, we have to let $C_1=0$ in \eqref{interm3}. We then go back to the profiles by undoing the change of variable \eqref{PSchange} and obtain from \eqref{interm2} that any profile contained in such orbits has $f(0)=A>0$, while \eqref{interm3} gives
$$
(f^{m-1})'(\xi)\sim\frac{(m-1)\alpha}{mN}\xi, \qquad {\rm as} \ \xi\to0, 
$$
which leads to the local behavior \eqref{beh.P0f} by integration on $(0,\xi)$ with $\xi>0$.
\end{proof}
It is now the turn of the critical point $P_1$. 
\begin{lemma}\label{lem.P1}
The phase portrait of the critical point $P_1$ depends on the critical exponents as follows:
\begin{itemize}
\item If $m\in(m_c,1)$, it is an unstable node. The profiles contained in the orbits stemming from $P_1$ present a vertical asymptote at $\xi=0$ with the form $f(\xi)\sim C\xi^{-(N-2)/m}$.
\item If $m\in(0,m_c)$ and $p<p_c(\sigma)$, it has a two-dimensional unstable manifold included in the invariant plane $\{X=0\}$ and a one-dimensional stable manifold included in the invariant plane $\{Z=0\}$.
\item If $m\in(0,m_c)$ and $p>p_c(\sigma)$, it has a two-dimensional stable manifold and a one-dimensional unstable manifold, the latter corresponding to the invariant $Y$-axis. In this case, the profiles contained in the orbits of the stable manifold present the fast decay \eqref{beh.P1} as $\xi\to\infty$.
\end{itemize}
\end{lemma}
Notice that the interesting case for our analysis is when the third item in Lemma \ref{lem.P1} holds true. We refrain from analyzing here the critical cases $m=m_c$ or $p=p_c(\sigma)$ where center manifolds appear. 
\begin{proof}
The linearization of the system \eqref{PSsyst} in a neighborhood of $P_1$ has the matrix 
$$
M(P_1)=\left(
         \begin{array}{ccc}
           \frac{mN-N+2}{m} & 0 & 0 \\[1mm]
           \frac{(N-2)(p_c(\sigma)-p)}{m(\sigma+2)} & N-2 & -1 \\[1mm]
           0 & 0 & \frac{(N-2)(p_c(\sigma)-p)}{m} \\
         \end{array}
       \right),
$$
with eigenvalues 
$$
\lambda_1=\frac{mN-N+2}{m}, \qquad \lambda_2=N-2, \qquad \lambda_3=\frac{(N-2)(p_c(\sigma)-p)}{m}
$$
and corresponding eigenvectors
$$
e_1=\left(\frac{(N-2m-2)(\sigma+2)}{(N-2)(p-p_c(\sigma))},1,0\right), \ e_2=(0,1,0), \ e_3=\left(0,\frac{m}{p(N-2)-m(\sigma+2)},1\right).
$$
We readily observe that, if $m\in(m_c,1)$ then necessarily $p<p_L(\sigma)<p_c(\sigma)$ by \eqref{interm1}, thus all three eigenvalues are positive, hence $P_1$ is an unstable node. On the orbits going out of $P_1$ we have $Y(\eta)\to-(N-2)/m$ as $\eta\to-\infty$, which easily leads after undoing the change of variable \eqref{PSchange} and an integration to the local behavior $f(\xi)\sim C\xi^{-(N-2)/m}$ as $\xi\to0$. If $m\in(0,m_c)$ but still $p<p_c(\sigma)$, then $\lambda_1<0$, while $\lambda_2>0$ and $\lambda_3>0$, hence we have a two-dimensional unstable manifold tangent to the plane spanned by the vectors $e_2$ and $e_3$. Since this is the plane $\{X=0\}$ which is invariant for the system \eqref{PSsyst}, we infer that the two-dimensional unstable manifold lies inside this plane. The one-dimensional unstable manifold tangent to the eigenvector $e_1$ belongs to the invariant plane $\{Z=0\}$ by the same argument. There are no profiles $f(\xi)$ corresponding to these orbits. Finally, if $m\in(0,m_c)$ and $p>p_c(\sigma)$, we have a two-dimensional stable manifold (and a one-dimensional unstable manifold contained in the invariant $Y$-axis) since $\lambda_1<0$ and $\lambda_3<0$. This stable manifold is tangent to the plane spanned by the eigenvectors $e_1$ and $e_3$ and contains orbits on which $Y(\eta)\to-(N-2)/m$ as $\eta\to\infty$. A standard argument of integration leads to the local behavior \eqref{beh.P1}.
\end{proof}
We are left with the critical point $P_2$, which exists only when $N\geq3$ and $p>p_c(\sigma)$. Notice that $P_1=P_2$ if $p=p_c(\sigma)$.
\begin{lemma}\label{lem.P2}
Let $p\in(p_c(\sigma),p_L(\sigma))$. Then the critical point $P_2$ is 
\begin{itemize}
\item a saddle point with a one-dimensional stable manifold and a two-dimensional unstable manifold if $p\in(p_c(\sigma),p_s(\sigma))$. In this case, the two-dimensional unstable manifold is contained in the invariant plane $\{X=0\}$. 
\item a stable focus or node if $p\in(p_s(\sigma),p_L(\sigma))$. 
\end{itemize}
In any of these two cases, the profiles contained in the stable manifold of the point $P_2$ present the slow decay \eqref{beh.P2} as $\xi\to\infty$.
\end{lemma}
Notice once more that we refrain from entering the analysis of this point in the critical case $p=p_s(\sigma)$. We shall see later that the two-dimensional unstable manifold contained in the invariant plane $\{X=0\}$ referred in the first item of Lemma \ref{lem.P2} is replaced by a center in this case. We will also avoid analyzing the exact ranges of $p$, $m$ and $\sigma$ for which the critical point is a focus or a node, since this is very tedious and it will not make any difference in the subsequent analysis.
\begin{proof}
The linearization of the system \eqref{PSsyst} in a neighborhood of the point $P_2$ has the matrix 
$$
M(P_2)=\left(
         \begin{array}{ccc}
           \frac{L}{p-m} & 0 & 0 \\[1mm]
           0 & -\frac{(N-2)(p-p_s(\sigma))}{p-m} & -1 \\[1mm]
           0 & \frac{(N-2)(\sigma+2)(p-p_c(\sigma))}{p-m} & 0 \\
         \end{array}
       \right),
$$
where $L$ is defined in \eqref{const.L}, and whose eigenvalues satisfy
$$
\lambda_1=\frac{L}{p-m}, \ \lambda_2+\lambda_3=-\frac{(N-2)(p-p_s(\sigma))}{p-m}, \ \lambda_2\lambda_3=\frac{(N-2)(\sigma+2)(p-p_c(\sigma))}{p-m}.
$$
Since $p_c(\sigma)<p<p_L(\sigma)$, we deduce that $\lambda_1<0$ and $\lambda_2\lambda_3>0$, while $\lambda_2+\lambda_3$ changes sign at $p=p_s(\sigma)$. Moreover, we can regard the matrix $M(P_2)$ as a block matrix and notice that the two eigenvectors corresponding to the eigenvalues $\lambda_2$ and $\lambda_3$ are contained in the invariant plane $\{X=0\}$. Thus, on the one hand, if $p<p_s(\sigma)$, either $\lambda_2$ and $\lambda_3$ are positive real numbers or conjugate complex numbers with positive real part, hence we have a saddle point with a single trajectory entering the point tangent to the direction of the eigenvector corresponding to $\lambda_1<0$ and a two-dimensional unstable manifold contained in the invariant plane $\{X=0\}$. On the other hand, if $p>p_s(\sigma)$, all three eigenvalues are either negative real numbers or (two of them) conjugate complex numbers with negative real part, thus $P_2$ is a stable node or focus. The only orbits containing profiles are the ones on the stable manifold of the point in any of the two cases. Such profiles satisfy 
$$
Z(\eta)\to\frac{(N-2)(\sigma+2)(p-p_c(\sigma))}{(p-m)^2}, \qquad {\rm as} \ \eta\to\infty,
$$
whence we immediately get the local behavior \eqref{beh.P2} as $\xi\to\infty$.
\end{proof}

\section{Critical points at infinity}\label{sec.inf}

The local analysis of any dynamical system is completed with the phase portraits and local manifolds of the critical points at infinity. To achieve this goal, we follow the theory in \cite[Section 3.10]{Pe} and pass to the Poincar\'e hypersphere by letting 
$$
X=\frac{\overline{X}}{W}, \qquad Y=\frac{\overline{Y}}{W}, \qquad Z=\frac{\overline{Z}}{W}
$$
and deducing that the critical points at space infinity, expressed in the new variables $(\overline{X},\overline{Y},\overline{Z},W)$, solve the following system
\begin{equation}\label{Poincare}
\left\{\begin{array}{ll}\frac{1}{\sigma+2}\overline{X}\overline{Y}[(p-m)\overline{X}-(\sigma+2)\overline{Y}]=0,\\
(p-1)\overline{X}\overline{Z}\overline{Y}=0,\\
\frac{1}{\sigma+2}\overline{Y}\overline{Z}[p(\sigma+2)\overline{Y}-(p-m)\overline{X}]=0,\end{array}\right.
\end{equation}
together with the condition of belonging to the equator of the hypersphere, which leads to $W=0$ and the additional equation $\overline{X}^2+\overline{Y}^2+\overline{Z}^2=1$, as given by \cite[Theorem 4, Section 3.10]{Pe}. Since we are considering only points with coordinates $\overline{X}\geq0$ and $\overline{Z}\geq0$, we find the following critical points on the Poincar\'e hypersphere:
\begin{equation}\label{crit.inf}
\begin{split}
&Q_1=(1,0,0,0), \ \ Q_{2,3}=(0,\pm1,0,0), \ \ Q_4=(0,0,1,0), \ \ Q_{\gamma}=\left(\gamma,0,\sqrt{1-\gamma^2},0\right),\\
&Q_5=\left(\frac{\sigma+2}{\sqrt{(\sigma+2)^2+(p-m)^2}},\frac{p-m}{\sqrt{(\sigma+2)^2+(p-m)^2}},0,0\right), 
\end{split}
\end{equation}
with $\gamma\in(0,1)$. We analyze these critical points below, leaving aside the point $Q_4$ that will be considered in a different section due to the employment of a different technique in its study.

\medskip 

\noindent \textbf{Projecting onto the $X$ variable: points $Q_1$, $Q_5$ and $Q_{\gamma}$.} For these points where the coordinate $\overline{X}$ is nonzero on the Poincar\'e hypersphere, we employ \cite[Theorem 5(a), Section 3.10]{Pe} and deduce that their analysis is topologically equivalent to the analysis of the (finite) critical points of the following system (where we replace the variable $w$ in the notation of \cite[Section 3.10]{Pe} by $x$ and put its equation on top)
\begin{equation}\label{PSinf1}
\left\{\begin{array}{ll}\dot{x}=x[(m-1)y-2x],\\
\dot{y}=-y^2+\frac{p-m}{\sigma+2}y+x-Nxy-xz,\\
\dot{z}=z[(p-1)y+\sigma x],\end{array}\right.
\end{equation}
where the variables expressed in lowercase letters are obtained from the original ones by the following change of variable 
\begin{equation}\label{change2}
x=\frac{1}{X}, \qquad y=\frac{Y}{X}, \qquad z=\frac{Z}{X},
\end{equation}
while the independent variable with respect to which derivatives are taken in \eqref{PSinf1} is defined implicitly by the differential equation
\begin{equation}\label{indep2}
\frac{d\eta_1}{d\xi}=\frac{\alpha}{m}\xi f^{1-m}(\xi).
\end{equation}
The system \eqref{PSinf1} is rather similar (varying only by a sign at one term) to systems considered and analyzed in previous works such as \cite{ILS23} and we will borrow some ideas from there. In this system, the critical points in \eqref{crit.inf} having a nonzero $\overline{X}$ coordinate are mapped into the following ones (obtained from their expressions in \eqref{crit.inf} by letting $x=0$ and dividing their $\overline{Y}$ and $\overline{Z}$ components by the $\overline{X}$ component): 
$$
Q_1=(0,0,0), \ Q_5=\left(0,\frac{p-m}{\sigma+2},0\right), \ Q_{\gamma}=(0,0,\kappa), \ {\rm where} \ \kappa:=\frac{\sqrt{1-\gamma^2}}{\gamma}, 
$$
noting that $\gamma\in(0,1)$ is mapped onto $\kappa\in(0,\infty)$. We next give the local analysis of the system \eqref{PSinf1} (and equivalently, \eqref{PSsyst}) in a neighborhood of these points.
\begin{lemma}\label{lem.Q1}
The critical point $Q_1$ has two-dimensional center manifolds and a one-dimensional unstable manifold. On any center manifold, the flow is unstable, thus $Q_1$ behaves like an unstable node. The orbits going out of $Q_1$ contain profiles with a vertical asymptote at the origin, more precisely, 
\begin{equation}\label{beh.Q1}
f(\xi)\sim C\xi^{-(\sigma+2)/(p-m)}, \qquad {\rm as} \ \xi\to0.
\end{equation}
\end{lemma}
\begin{proof}
The linearization of the system \eqref{PSinf1} in a neighborhood of the critical point $Q_1=(0,0,0)$ has the matrix 
$$
M(Q_1)=\left(
         \begin{array}{ccc}
           0 & 0 & 0 \\[1mm]
           1 & \frac{p-m}{\sigma+2} & 0 \\
           0 & 0 & 0 \\
         \end{array}
       \right),
$$
thus, we have to analyze the flow on the two-dimensional center manifolds. To this end, we replace $y$ by the new variable
$$
w:=\frac{p-m}{\sigma+2}y+x
$$
and the system \eqref{PSinf1} becomes
\begin{equation}\label{PSinf1.bis}
\left\{\begin{array}{ll}\dot{x}=\frac{1}{\beta}x^2+\frac{(m-1)\alpha}{\beta}xw,\\[1mm]
\dot{w}=\frac{\beta}{\alpha}w-\frac{\alpha}{\beta}w^2-\frac{\beta}{\alpha}xz+\frac{(m+1)\alpha-N\beta}{\beta}xw-\frac{m\alpha-(N-2)\beta}{\beta}x^2,\\[1mm]
\dot{z}=\frac{1}{\beta}xz+\frac{(p-1)\alpha}{\beta}zw,\end{array}\right.
\end{equation}
where $\alpha$ and $\beta$ are defined in \eqref{ss.exponents}. This system has the canonical form in which we can apply the local center manifold theorem and look for a Taylor expansion of the center manifold. We thus write the center manifold, according to \cite[Theorem 3, Section 2.5]{Carr}, in the form
$$
w=h(x,z)=ax^2+bxz+cz^2+o(|(x,z)|^2),
$$
where $a$, $b$, $c$ are to be determined. The center manifold equation is a rather tedious one (see for example \cite[Theorem 1, Section 2.12]{Pe}) and we omit it here, but if we neglect the terms of order at least three, we are left with some easy identifications only related to the terms not containing $w$ in the equation of $\dot{w}$ in the system \eqref{PSinf1.bis}. We then find that $c=0$, $b=1$ and 
$$
a=\frac{(\sigma+2)(N-2)(p_c(\sigma)-p)}{(p-m)^2}.
$$
The flow on any center manifold is then given locally, according to \cite[Theorem 2, Section 2.4]{Carr} by the reduced system obtained from the equations of $\dot{x}$ and $\dot{z}$ in the system \eqref{PSinf1.bis} by replacing $w$ with the approximation of the center manifold. Noticing that the terms involving $w$ in these equations are at least cubic, the reduced system writes locally as 
\begin{equation}\label{interm4}
\left\{\begin{array}{ll}\dot{x}&=\frac{1}{\beta}x^2+x^2O(|(x,z)|),\\[1mm]
\dot{z}&=\frac{1}{\beta}xz+xO(|(x,z)|^2),\end{array}\right.
\end{equation}
where we have used the fact that $w$ itself is a multiple of $x$ in the quadratic approximation (and it can be shown by induction that in fact $w/x$ is a polynomial on the center manifold). We readily get from \eqref{interm4} that any center manifold is in reality unstable, matching with the sign of the only nonzero eigenvalue of $M(Q_1)$ and we infer from \cite[Lemma 1, Section 2.4]{Carr} (applied for the reversed direction of the flow) that $Q_1$ behaves like an unstable node of the system \eqref{PSinf1}. Moreover, the same system \eqref{interm4} gives by integration that $z(\eta_1)\sim Cx(\eta_1)$ as $\eta_1\to-\infty$ on any center manifold, for some $C>0$, which in the original variables gives $Z(\eta)\to C$ as $\eta\to-\infty$ and then the claimed local behavior \eqref{beh.Q1}.  
\end{proof}
It is now the turn for the point $Q_5$.
\begin{lemma}\label{lem.Q5}
The critical point $Q_5$ is a saddle point having a two-dimensional stable manifold fully contained in the invariant plane $\{z=0\}$ and a one-dimensional manifold contained in the invariant plane $\{x=0\}$ of the system \eqref{PSinf1}. The orbits connecting to it do not contain any profiles $f(\xi)$.
\end{lemma}
\begin{proof}
The linearization of the system \eqref{PSinf1} in a neighborhood of $Q_5$ has the matrix 
$$
M(Q_5)=\left(
         \begin{array}{ccc}
           -\frac{(1-m)(p-m)}{\sigma+2} & 0 & 0 \\[1mm]
           1-\frac{N(p-m)}{\sigma+2} & -\frac{p-m}{\sigma+2} & 0 \\[1mm]
           0 & 0 & \frac{(p-1)(p-m)}{\sigma+2} \\
         \end{array}
       \right),
$$
with two negative eigenvalues and one positive eigenvalue. An inspection of the corresponding eigenvectors shows that the two eigenvectors corresponding to the negative eigenvalues lie both in the invariant plane $\{z=0\}$, and thus the whole stable manifold, while the unique unstable orbit is contained in the vertical line $\{y=(p-m)/(\sigma+2)\}$ inside the plane $\{x=0\}$.
\end{proof}
The analysis of the critical points $Q_{\gamma}$ will not give any further news. 
\begin{lemma}\label{lem.Qg}
For any $\gamma\in(0,1)$ (and equivalently $\kappa\in(0,\infty)$), all the orbits connecting to the points $Q_{\gamma}$ are fully contained in the invariant plane $\{x=0\}$. 
\end{lemma}
\begin{proof}
We omit here the detailed proof, as it goes following exactly the same lines as in, for example, \cite[Lemma 2.4]{ILS23} or \cite[Lemma 2.3]{IMS23} where the calculations are given with details. The main point is that, when following the choice of the parameter $\kappa\in(0,\infty)$ for which orbits not contained in the plane $\{x=0\}$ might exist, as explained in the two quoted references, we get that 
$$
\frac{\sqrt{1-\gamma^2}}{\gamma}=\kappa=\frac{L}{(\sigma+2)(p-1)}<0,
$$ 
recalling that the constant $L$ defined in \eqref{const.L} is negative for $p<p_L(\sigma)$, and we reach a contradiction with the fact that $\gamma>0$.
\end{proof}

\medskip

\noindent \textbf{Projecting onto the $Y$ variable: the points $Q_2$ and $Q_3$}. We analyze in this paragraph the critical points on the Poincar\'e hypersphere such that $\overline{X}=0$ but $\overline{Y}\neq0$, that is, $Q_2$ and $Q_3$. We infer from \cite[Theorem 5(b), Section 3.10]{Pe} that the phase portrait in a neighborhood of them is topologically equivalent to the one in a neighborhood of the origin of the following system 
\begin{equation}\label{PSinf2}
\left\{\begin{array}{lll}\pm\dot{x}=-x-Nxw+\frac{p-m}{\sigma+2}x^2+x^2w-xzw,\\[1mm]
\pm\dot{z}=-pz-(N+\sigma)zw+\frac{p-m}{\sigma+2}xz+xzw-z^2w,\\[1mm]
\pm\dot{w}=-mw-(N-2)w^2+\frac{p-m}{\sigma+2}xw+xw^2-zw^2,\end{array}\right.
\end{equation} 
where the new variables $x$, $z$, $w$ are obtained from the original variables of the system \eqref{PSsyst} by the following change of variables:
\begin{equation}\label{change3}
x=\frac{X}{Y}, \qquad z=\frac{Z}{Y}, \qquad w=\frac{1}{Y}.
\end{equation}
\begin{lemma}\label{lem.Q23}
The critical points $Q_2$ and $Q_3$ are, respectively, an unstable node and a stable node. The orbits going out of $Q_2$ correspond to profiles $f(\xi)$ such that there exists $\xi_0\in(0,\infty)$ and $\delta>0$ for which
\begin{equation}\label{beh.Q2}
f(\xi_0)=0, \qquad (f^m)'(\xi_0)=C>0, \qquad f>0 \ {\rm on} \ (\xi_0,\xi_0+\delta),
\end{equation}
while the orbits entering the stable node $Q_3$ correspond to profiles $f(\xi)$ such that there exists $\xi_0\in(0,\infty)$ and $\delta\in(0,\xi_0)$ for which
\begin{equation}\label{beh.Q3}
f(\xi_0)=0, \qquad (f^m)'(\xi_0)=-C<0, \qquad f>0 \ {\rm on} \ (\xi_0-\delta,\xi_0).
\end{equation}
\end{lemma}
Notice that these profiles do not give rise to well-defined non-negative self-similar solutions, since the regularity at the vanishing point $\xi_0\in(0,\infty)$ is not sufficient. It is a well-established fact in the theory of the porous medium equation that, in order for a profile to be extended with zero starting from the vanishing point and produce in this way a weak solution, one needs to have the interface condition $(f^m)'(\xi_0)=0$, see \cite[Section 9.8]{VPME}.
\begin{proof}
We have to choose in the system \eqref{PSinf2} the signs minus for the analysis in a neighborhood of $Q_2$ and plus in a neighborhood of $Q_3$, as indicated by the direction of the flow, while the linearization of the vector field of the system \eqref{PSinf2} in a neighborhood of the origin has the matrix 
$$
M(Q_{2,3})=\left(
             \begin{array}{ccc}
               -1 & 0 & 0 \\
               0 & -p & 0 \\
               0 & 0 & -m \\
             \end{array}
           \right).
$$
Thus, $Q_2$ is an unstable node (since we have chosen the minus sign in front of the derivatives in \eqref{PSinf2}) while $Q_3$ is a stable node. The analysis of the profiles follows the same lines as the one performed in full detail in \cite[Lemma 2.6]{IS21} or \cite[Lemma 2.4]{IS22} and will be omitted here.
\end{proof}
We are only left with the analysis of the critical point $Q_4$. Since projecting on the $Z$ component of the hypersphere leads to a system whose origin has all three eigenvalues equal to zero, we will employ a different approach that is given in the next section. 

\section{The critical point $Q_4$}\label{sec.top}

We begin with a lemma limiting the possible local behaviors of all self-similar profiles. Its proof is technical and lengthy and will be given in the Appendix which ends the paper.
\begin{lemma}\label{lem.Q4}
There are no solutions to Eq. \eqref{ODE.forward} satisfying simultaneously the following limits 
\begin{equation}\label{limits}
\begin{split}
&\xi^{\sigma}f(\xi)^{p-1}\to\infty, \qquad \xi^{\sigma+2}f(\xi)^{p-m}\to\infty, \qquad \xi^{\sigma-2}f(\xi)^{m+p-2}\to\infty, \\
&\xi^{\sigma+1}f(\xi)^{p-m+1}(f')^{-1}(\xi)\to\pm\infty,
\end{split}
\end{equation}
the limits above being taken in any of the possible cases as $\xi\to0$, $\xi\to\xi_0\in(0,\infty)$ or $\xi\to\infty$.
\end{lemma}
If we translate this result in terms of the variables of the system \eqref{PSsyst}, Lemma \ref{lem.Q4} gives that there are no orbits in the phase space (with any orientation) satisfying simultaneously the following conditions:
\begin{equation}\label{interm21bis}
Z\to\infty, \qquad \frac{Z}{X}\to\infty, \qquad \frac{Z}{Y}\to\pm\infty, \qquad \frac{Z}{X^2}\to\infty.
\end{equation}
If we analyze the critical point $Q_4$ on the Poincar\'e hypersphere, we observe that it is characterized by the first three conditions in \eqref{interm21bis}, and that in the invariant plane $\{X=0\}$ no orbit can connect to it, since \eqref{interm21bis} is authomatically fulfilled. We are thus left to consider orbits with $X>0$, including possibly $X\to\infty$ (where we shall precisely see that some orbits appear), thus it is better to go back to our second system \eqref{PSinf1} and notice that \eqref{interm21bis} is transformed through \eqref{change2} into the following list of conditions:
\begin{equation}\label{interm21}
z\to\infty, \qquad \frac{z}{x}\to\infty, \qquad \frac{z}{y}\to\pm\infty, \qquad xz\to\infty.
\end{equation}
The latter limit gives then the idea of setting $w:=xz$ in the system \eqref{PSinf1} and obtain the new system 
\begin{equation}\label{PSsyst3}
\left\{\begin{array}{lll}\dot{x}=x((m-1)y-2x),\\[1mm]
\dot{y}=-y^2+\frac{p-m}{\sigma+2}y+x-Nxy-w,\\[1mm]
\dot{w}=w((\sigma-2)x+(m+p-2)y),\end{array}\right.
\end{equation}
in which we will only look for possible new orbits at points with $w$ finite, according to Lemma \ref{lem.Q4} and \eqref{interm21}. It is easy to check that the only such critical points are the origin $Q_1'=(0,0,0)$ and the point $Q_5'=(0,(p-m)/(\sigma+2),0)$, which correspond to the critical points $Q_1$ and $Q_5$ analyzed in Lemma \ref{lem.Q1} and \ref{lem.Q5}, but we use the notation with primes since we refer to these points as critical in the system \eqref{PSsyst3}. If some new orbits are to be found somewhere, the previous discussion shows that they should appear necessarily at these two points. As a special case, notice that when $m+p=2$ there also exists a critical parabola given by
\begin{equation}\label{interm33}
x=0, \qquad -y^2+\frac{p-m}{\sigma+2}y-w=0.
\end{equation}
\begin{lemma}\label{lem.Q4a}
If $m+p>2$, the critical point $Q_5'$ is a saddle point having a two-dimensional stable manifold fully contained in the invariant plane $\{z=0\}$ and a one-dimensional unstable manifold contained in the invariant plane $\{x=0\}$. If $m+p<2$, $Q_5'$ is a stable node and the profiles contained in its orbits have a vertical asymptote from the left at some point $\xi_0\in(0,\infty)$ with the local behavior
\begin{equation}\label{beh.Q4a}
f(\xi)\sim\left[D-\frac{(1-m)\beta}{2m}\xi^2\right]^{-1/(1-m)}, \qquad D=\frac{(1-m)\beta}{2m}\xi_0^2.
\end{equation}
\end{lemma}
\begin{proof}
The linearization of the system \eqref{PSsyst3} in a neighborhood of the point $Q_5'$ has the matrix
$$
M(Q_5')=\left(
          \begin{array}{ccc}
            -\frac{(1-m)(p-m)}{\sigma+2} & 0 & 0 \\[1mm]
            1-\frac{N(p-m)}{\sigma+2} & -\frac{p-m}{\sigma+2} & -1 \\[1mm]
            0 & 0 & \frac{(m+p-2)(p-m)}{\sigma+2} \\
          \end{array}
        \right),
$$
and it is then rather clear that, if $m+p-2>0$, the local analysis is completely similar to the one performed in Lemma \ref{lem.Q5} and does not bring anything new. On the contrary, if $m+p<2$, the three eigenvalues of $M(Q_5')$ are negative, thus we find a stable node. There exist then non-trivial profiles contained in the orbits entering the point. Since $x(\eta_1)\to0$ and $y(\eta_1)\to(p-m)/(\sigma+2)$ as $\eta_1\to\infty$ on the orbits entering $Q_5'$ (where we recall that $\eta_1$ is the independent variable of the systems \eqref{PSinf1} and \eqref{PSsyst3}, defined implicitly in \eqref{indep2}), we infer by integration from the first equation in \eqref{PSinf1} that, as $\eta_1\to\infty$,
\begin{equation}\label{interm27}
\begin{split}
X(\eta_1)&\sim\frac{(p-m)(1-m)}{C(p-m)(1-m)\exp{[(p-m)(1-m)\eta_1/(\sigma+2)]}-2(\sigma+2)}\\&\sim Ce^{-(p-m)(1-m)\eta_1/(\sigma+2)},
\end{split}
\end{equation}
on the orbits entering $Q_5'$. We then obtain by inverting \eqref{indep2} and employing \eqref{interm27} that 
\begin{equation*}
\begin{split}
\xi(\eta_1)&=\frac{m}{\alpha}\int_0^{\eta_1}\frac{1}{sf^{1-m}(s)}\,ds=\int_0^{\eta_1}sX(s)\,ds\\
&\to\int_0^{\infty}\eta_1X(\eta_1)\,d\eta_1\sim C\int_0^{\infty}\eta_1e^{-(p-m)(1-m)\eta_1/(\sigma+2)}\,d\eta_1<\infty,
\end{split}
\end{equation*}
as $\eta_1\to\infty$, whence we conclude that the profiles contained in orbits entering $Q_5'$ end up as $\xi\to\xi_0\in(0,\infty)$. Taking then into account that $y(\eta_1)\to(p-m)/(\sigma+2)$ as $\eta_1\to\infty$, we obtain by undoing \eqref{change2} and \eqref{PSchange} that
$$
(f^{m-1})'(\xi)\sim\frac{(m-1)\beta}{m}\xi, \qquad {\rm as} \ \xi\to\xi_0\in(0,\infty),
$$
on any profile entering $Q_5'$. A simple integration leads to the local behavior \eqref{beh.Q4a} as $\xi\to\xi_0\in(0,\infty)$.
\end{proof}
The analysis of the remaining (and last) critical point $Q_1'$ is more involved. 
\begin{lemma}\label{lem.Q4b}
The non-hyperbolic critical point $Q_1'$ has a three-dimensional unstable sector and a two-dimensional stable sector. The orbits on its unstable sector are the same ones inherited from the unstable node $Q_1$ in the system \eqref{PSinf1}. The orbits on its stable sector contain profiles with a vertical asymptote from the left at some point $\xi_0\in(0,\infty)$ with the local behavior
\begin{equation}\label{beh.Q4b}
f(\xi)\sim\left[D-\frac{p-1}{\sigma\beta}\xi^{\sigma}\right]^{-1/(p-1)}, \qquad D=\frac{p-1}{\sigma\beta}\xi_0^{\sigma}.
\end{equation}
\end{lemma}
\begin{proof}
We have to analyze the center manifold of the critical point $Q_1'=(0,0,0)$ in the system \eqref{PSsyst3}. To this end, we set first 
$$
t=\frac{p-m}{\sigma+2}y+x-w, 
$$
and obtain the new system in variables $(x,t,w)$:
\begin{equation}\label{PSsyst3bis}
\left\{\begin{array}{lll}\dot{x}=\frac{1}{\beta}[x^2-\alpha(1-m)xw-\alpha(1-m)xt],\\[1mm]
\dot{t}=\frac{\beta}{\alpha}t-\frac{\alpha}{\beta}t^2-\frac{N\beta-\alpha(m+1)}{\beta}tx-\frac{\alpha(m+p)}{\beta}tw+Ax^2-Bxw-Cw^2,\\[1mm]
\dot{w}=\frac{1}{\beta}[2xw+(m+p-2)\alpha tw+(m+p-2)\alpha w^2],\end{array}\right.
\end{equation}
where 
\begin{equation}\label{interm22}
A=\frac{(N-2)\beta-m\alpha}{\beta}, \ B=\frac{(N-2+\sigma)\beta-(2m+p-1)\alpha}{\beta}, \ C=\frac{(m+p-1)\alpha}{\beta}.
\end{equation}
A similar analysis to the one performed in Lemma \ref{lem.Q1} employing the center manifold theory in \cite{Carr} gives that the center manifold has the following second order Taylor approximation 
\begin{equation}\label{center.man}
t=-\frac{A\alpha}{\beta}x^2+\frac{B\alpha}{\beta}xw+\frac{C\alpha}{\beta}w^2+o(|(x,w)|^2),
\end{equation}
where $A$, $B$, $C$ are defined in \eqref{interm22}, and the flow on the center manifold is given by the following reduced system 
\begin{equation}\label{syst.red}
\left\{\begin{array}{ll}\dot{x}=\frac{1}{\beta}x[x-\alpha(1-m)w],\\[1mm]
\dot{w}=\frac{1}{\beta}w[2x+\alpha(m+p-2)w].\end{array}\right. 
\end{equation}
We readily remark that, if $m+p>2$, then $\dot{w}>0$ on the trajectories of the system \eqref{syst.red}, thus the flow on the center manifold is unstable. On the contrary, if $m+p<2$, an analysis completely similar (and with exactly the same calculations) as the one done in \cite[Section 2]{IS20} and based on the classification of the two-homogeneous dynamical system given in \cite{Date79} leads to the classification of the sectors and to the local behavior \eqref{beh.Q4b}. For the sake of completeness, we sketch here a more formal argument for the proof. We find by integration that the curves of the system \eqref{syst.red} satisfy
\begin{equation}\label{curves}
\frac{w^{p-1}}{x^{m+p-2}(x+\alpha(p-1)w)^{p-m}}=C>0, 
\end{equation}
and it is easy to check that there are no curves tangent to some line $w=kx$ with $k>0$ as $(x,w)\to(0,0)$. Considering the isoclines 
$$
r_1: x-\alpha(1-m)w=0, \qquad r_2: 2x-\alpha(2-m-p)w=0,
$$
the latter only for the range $m+p<2$, these lines split the flow on the center manifold into three sectors (or only two if $m+p\geq2$): 
\begin{equation*}
\begin{split}
&S_1:=\{(x,w):0\leq\alpha(1-m)w<x\},\\ &S_2:=\left\{(x,w):0<\frac{\alpha(2-m-p)}{2}w<x<\alpha(1-m)w\right\},\\ &S_3:=\left\{(x,w):0<x<\frac{\alpha(2-m-p)}{2}w\right\},
\end{split}
\end{equation*}
noticing that $S_3=\emptyset$ if $m+p\geq2$. It is then obvious that $S_1$ is an unstable sector since $\dot{x}>0$, $\dot{w}>0$ in the system \eqref{syst.red}, $S_2$ is a saddle sector and $S_3$ (which is nonempty only when $m+p<2$) is a stable sector. Thus, the orbits go out of $Q_1'$ through the sector $S_1$ and we infer from the fact that no linear dependence is allowed, that $w/x\to0$ on these trajectories as the orbit approaches $Q_1'$. Going back to the variables of the system \eqref{PSinf1}, we find that $z=w/x\to0$, thus all the orbits going out of $Q_1'$ are in fact contained in the orbits going out of the initial point $Q_1$ in variables $(x,y,z)$ analyzed in Lemma \ref{lem.Q1}. 

Assume now that $m+p<2$. We are thus left with the new stable sector $S_3$ of the point $Q_1'$. Since no trajectory can enter $Q_1'$ tangent to a line (even on subsequences), we conclude that $w/x\to\infty$ on the trajectories in this sector, hence $z\to\infty$ and we have indeed orbits characteristic to the critical point $Q_4$. We then deduce from \eqref{curves} that 
\begin{equation}\label{curves2}
Cx\sim\frac{z}{(\alpha(p-1)z+1)^{(p-m)/(p-1)}}, \qquad {\rm as} \ z\to\infty,
\end{equation}
along the trajectories entering $Q_1'$. We recall that the variables $x$, $y$, $z$ of the system \eqref{PSinf1} are expressed in terms of profiles by 
$$
x(\xi)=\frac{m}{\alpha}\xi^{-2}f(\xi)^{m-1}, \ \ y(\xi)=\frac{m}{\alpha}\xi^{-1}f(\xi)^{m-2}f'(\xi), \ \ z(\xi)=\frac{1}{\alpha}\xi^{\sigma}f(\xi)^{p-1},
$$
hence the asymptotic behavior \eqref{curves2} becomes in terms of profiles 
\begin{equation}\label{curves3}
\begin{split}
C&\sim\frac{\xi^{\sigma+2}f(\xi)^{p-m}}{\left[(\alpha(p-1)\xi^{\sigma}f(\xi)^{p-1}+1)^{(p-m)/(p-1)}\right]}\\
&=\xi^{L/(p-1)}\left[\frac{\xi^{\sigma}f(\xi)^{p-1}}{\alpha(p-1)\xi^{\sigma}f(\xi)^{p-1}+1}\right]^{(p-m)/(p-1)}\\
&=\xi^{L/(p-1)}\left[\frac{\alpha z(\xi)}{\alpha^2(p-1)z(\xi)+1}\right]^{(p-m)/(p-1)}\sim\left[\frac{1}{\alpha(p-1)}\right]^{(p-m)/(p-1)}\xi^{-L/(p-1)},
\end{split}
\end{equation}
as $z\to\infty$, where $L<0$ is the constant defined in \eqref{const.L}. It thus follows from \eqref{curves3} that these profiles have to blow up as $\xi\to\xi_0$ for some constant $\xi_0\in(0,\infty)$. Moreover, recalling the equation \eqref{center.man} of the center manifold and the definition of $t$, we infer that 
$$
\frac{p-m}{\sigma+2}y-w=cw^2-x+o(w^2)=o(w),
$$
and passing to profiles we find that 
$$
\frac{\beta}{\alpha}\xi^{-1}f^{m-2}(\xi)f'(\xi)\sim\frac{1}{\alpha}\xi^{\sigma-2}f(\xi)^{m+p-2}, \qquad {\rm as} \ \xi\to\xi_0,
$$
which readily gives the claimed local behavior \eqref{beh.Q4b}. As we have said, a fully rigorous argument must use the classification of the two-homogeneous systems given by \cite{Date79} and the main calculations can be found in \cite[Section 2]{IS20}.
\end{proof}

\noindent \textbf{Remark}. The outcome of Lemmas \ref{lem.Q4a} and \ref{lem.Q4b} shows that in reality, the critical point $Q_4$ brings us two new stable manifolds whose orbits contain profiles with vertical asymptotes as given by \eqref{beh.Q4a} and \eqref{beh.Q4b}. However, they will not be important for the subsequent analysis, the only thing that the new manifolds are all stable being in fact sufficient. We also refrain to enter the details of the local analysis of the critical parabola \eqref{interm33} and just observe again that the new manifolds are also either stable or contained in the invariant plane $\{x=0\}$, a detailed analysis of them can be done similarly with the one in \cite[Lemma 2.4]{IMS22}.

\section{The bifurcation at $p=p_s(\sigma)$: the plane $\{X=0\}$}\label{sec.invariant}

In this section we perform the analysis of the invariant plane $\{X=0\}$ of the system \eqref{PSsyst}. The critical points are the same ones $P_0$, $P_1$, $P_2$, $Q_2$ and $Q_3$ already analyzed, but we shall see that the critical exponent $p_s(\sigma)$ realizes a bifurcation in the analysis which is decisive for the proof of both Theorems \ref{th.forward} and \ref{th.extinction}. Let us begin with the reduced system of \eqref{PSsyst} in the plane $\{X=0\}$
\begin{equation}\label{syst.x0}
\left\{\begin{array}{ll}\dot{Y}=-(N-2)Y-mY^2-Z,\\ \dot{Z}=(\sigma+2)Z+(p-m)YZ.\end{array}\right.
\end{equation}
The following result already illustrates how $p_s(\sigma)$ comes into play. 
\begin{lemma}\label{lem.explicit}
If $p\neq p_s(\sigma)$, the system \eqref{syst.x0} does not admit limit cycles. If $p=p_s(\sigma)$, there exists an explicit trajectory of \eqref{syst.x0}
\begin{equation}\label{cylinder}
Z=-\frac{N+\sigma}{N-2}(mY+N-2)Y, 
\end{equation}
connecting the critical points $P_0$ and $P_1$. Moreover, $P_2$ is a center if $p=p_s(\sigma)$.
\end{lemma}
\begin{proof}
The easiest way to proceed is to perform a new change of variable in \eqref{syst.x0} by setting 
\begin{equation}\label{varT}
T:=\sigma+2+(p-m)Y
\end{equation}
and observe that the system \eqref{syst.x0} is transformed into the following one
\begin{equation}\label{interm23}
\left\{\begin{array}{ll}\dot{T}=-\frac{(N-2)(p-p_s(\sigma))}{p-m}T-\frac{m}{p-m}T^2-(p-m)Z+\frac{(N-2)(p-p_c(\sigma))(\sigma+2)}{p-m},\\ \dot{Z}=TZ.\end{array}\right.
\end{equation}
Notice first that the system \eqref{interm23} simplifies strongly if $p=p_s(\sigma)$ and can be integrated explicitly to find its general curves
\begin{equation}\label{curves4}
T^2=\frac{1}{N-2}\left[(N-2)(\sigma+2)^2-\frac{4m(\sigma+2)^2}{N+\sigma}Z+C(N-2)Z^{-(N-2)/(N+2)}\right],
\end{equation}
with $C\in\real$ integration constant. As we are interested in orbits passing through $Z=0$, we have to let $C=0$ in \eqref{curves4} and thus get the explicit orbit 
\begin{equation}\label{interm24}
T^2=(\sigma+2)^2\left[1-\frac{4m}{(N-2)(N+\sigma)}Z\right]. 
\end{equation}
Undoing the change of variable defining $T$ and performing straightforward algebraic manipulations, we find that the trajectory \eqref{interm24} is expressed as \eqref{cylinder} in the initial variables $(Y,Z)$, as claimed. It is obvious that this orbit contains the points $P_0=(0,0)$ and $P_1=(-(N-2)/m,0)$. 

In order to prove that there are no limit cycles in the system \eqref{syst.x0}, we employ the Dulac Criterion \cite[Theorem 2, Section 3.9]{Pe}. More precisely, we go back to the system \eqref{interm23} and we compute the divergence of its vector field multiplied by the ``integrating factor" $Z^a$ with $a\in\real$ to be determined later. We find that
$$
D(T,Z):=\frac{\partial(Z^a\dot{T})}{\partial T}+\frac{\partial(Z^a\dot{Z})}{\partial Z}=-\frac{(N-2)(p-p_s(\sigma))}{p-m}Z^a-\frac{2m}{p-m}Z^aT+(a+1)Z^aT.
$$ 
Choosing then $a=(3m-p)/(p-m)$, we simplify the last two terms on the right hand side of the previous equality and get that
$$
D(T,Z)=-\frac{(N-2)(p-p_s(\sigma))}{p-m}Z^a
$$
which has a sign if $p\neq p_s(\sigma)$. It follows that there are no limit cycles if $p\neq p_s(\sigma)$. Finally, we notice that, if $p=p_s(\sigma)$, the system \eqref{interm23} is symmetric to the change $T\mapsto-T$, which proves that $P_2$ is a center. Let us remark that the explicit curves \eqref{curves4} with $C<0$ are periodic orbits surrounding $P_2$.
\end{proof}
The next technical result completes the analysis of the phase plane associated to the invariant plane $\{X=0\}$, emphasizing on the bifurcation that takes place at $p=p_s(\sigma)$.
\begin{lemma}\label{lem.connect}
Assume that $p>1$ and $N\geq3$. Then the orbit going out of the critical point $P_0$ in the plane $\{X=0\}$ connects to $Q_3$ if $1<p<p_s(\sigma)$, while the orbit entering $P_1$ arrives from $P_2$ provided $p\in(p_c(\sigma),p_s(\sigma))$. On the contrary, if $p>p_s(\sigma)$, then the orbit going out of $P_0$ enters $P_2$, while the orbit entering $P_1$ arrives from the critical point $Q_2$.
\end{lemma}
\begin{proof}
Let us consider the curve in the plane $\{X=0\}$ given by the explicit solution \eqref{cylinder}. We have already noticed that this is the unique orbit connecting the critical points $P_0$ and $P_1$ (which are both saddle points) in the plane $\{X=0\}$ for $p=p_s(\sigma)$, but geometrically we can consider it as a curve in the invariant plane $\{X=0\}$ for any $p>1$. The direction of the flow of the system \eqref{syst.x0} over this curve is given by the sign of the expression 
$$
H(Y)=\frac{(N+\sigma)(p_s(\sigma)-p)}{N-2}Y^2(mY+N-2)\left\{\begin{array}{ll} >0, \ {\rm if} \ p<p_s(\sigma),\\<0, \ {\rm if} \ p>p_s(\sigma),\end{array}\right. 
$$
since $-(N-2)/m\leq Y\leq 0$ on this curve. Notice then that the critical point $P_2$ lies inside the region limited by the curve \eqref{cylinder} and the $Y$-axis. Indeed, if we evaluate the curve \eqref{cylinder} at $Y=-(\sigma+2)/(p-m)$, we get 
$$
Z=\frac{(p-p_c(\sigma))(\sigma+2)(N+\sigma)}{(p-m)^2}\geq\frac{(N-2)(p-p_c(\sigma))(\sigma+2)}{(p-m)^2}=Z(P_2). 
$$
We further observe that the unique orbit entering the saddle point $P_1$ in the system \eqref{syst.x0} enters tangent to the eigenvector 
$$
e_2=\left(1,\frac{p(N-2)-m(\sigma+2)}{m}\right),
$$
whose slope varies in an increasing way with respect to $p$. Thus, since we already know that this orbit coincides with the curve \eqref{cylinder} if $p=p_s(\sigma)$ and that \eqref{cylinder} does not depend on $p$, we readily obtain by monotonicity that the orbit enters $P_1$ through the interior region to the curve \eqref{cylinder} if $p\in(p_c(\sigma),p_s(\sigma))$ and through the exterior region if $p>p_s(\sigma)$. Finally, we have to spot the monotonicity with respect to $p$ of the orbit starting from $P_0$ in the system \eqref{syst.x0}. To this end, since the eigenvector $e_1=(-1,N+\sigma)$ does not depend on $p$, we have to go to the second order of approximation of the orbit. Following \cite{Shilnikov}, we look for a quadratic approximation in the form 
\begin{equation}\label{interm25}
Z=-(N+\sigma)Y+KY^2+o(Y^2), 
\end{equation}
with $K\in\real$ to be determined. We compute the approximation of $dY/dZ$ in a neighborhood of $P_0$ in two ways: once from the ansatz \eqref{interm25} directly and one from the system \eqref{syst.x0} by replacing then $Z$ with the ansatz \eqref{interm25}. We thus find that 
\begin{equation*}
\begin{split}
-(N+\sigma)+2KY&+o(Y)=\frac{dZ}{dY}=\frac{(\sigma+2)Z+(p-m)YZ}{-(N-2)Y-Z-mY^2}\\
&=\frac{[\sigma+2+(p-m)Y][-(N+\sigma)Y+KY^2+o(Y^2)]}{-(N-2)Y-mY^2+(N+\sigma)Y-KY^2+o(Y^2)}\\
&=\frac{-(\sigma+2)(N+\sigma)Y+[K(\sigma+2)-(N+\sigma)(p-m)]Y^2+o(Y^2)}{(\sigma+2)Y-(K+m)Y^2+o(Y^2)},
\end{split}
\end{equation*}
which leads after straightforward calculations to $K=-(N+\sigma)p/(N+2\sigma+2)$. We thus obtain the local second order approximation
\begin{equation}\label{interm26}
Z=-(N+\sigma)Y-\frac{(N+\sigma)p}{N+2\sigma+2}Y^2+o(Y^2).
\end{equation}
We thus remark that the orbit \eqref{interm26} varies in a decreasing way with respect to $p$. Since again, for $p=p_s(\sigma)$, this orbit identifies with the explicit curve \eqref{cylinder} which does not depend on $p$, it follows that the curve \eqref{interm26} goes out of $P_0$ into the interior region to the curve \eqref{cylinder} if $p>p_s(\sigma)$ and into the region exterior to the curve \eqref{cylinder} if $p\in(p_c(\sigma),p_s(\sigma))$. Since by Lemma \ref{lem.explicit} the system \eqref{syst.x0} does not admit any limit cycle if $p\neq p_s(\sigma)$, the direction of the flow given by the sign of $H(Y)$ over the curve \eqref{cylinder} and the Poincar\'e-Bendixon Theory (see for example \cite[Section 3.7]{Pe}) imply that the orbit starting from $P_0$ remains forever in the region limited by the curve \eqref{cylinder} if $p>p_s(\sigma)$ and has to enter the critical point $P_2$ (which is a stable focus or node), while the same orbit stays forever in the exterior region to the curve \eqref{cylinder} and has to connect to the stable node $Q_3$ if $p\in(p_c(\sigma),p_s(\sigma))$, since the orbit entering $P_1$ comes from the interior region to the curve (and thus from $P_2$ which is an unstable focus if $p<p_s(\sigma)$). Finally, if $p>p_s(\sigma)$, the orbit entering $P_1$ has to come from a critical point lying in the exterior region to the curve \eqref{cylinder} and it is easy to see by elimination that the only possible candidate is $Q_2$. This ends the proof.
\end{proof}

\section{Proof of Theorem \ref{th.forward}}\label{sec.global}

With all these preparations, we are now in a position to complete the proof of Theorem \ref{th.forward}. 
\begin{proof}[Proof of Theorem \ref{th.forward}]
\textbf{Part 2} of Theorem \ref{th.forward} follows easily from Lemma \ref{lem.connect}. Indeed, let $p>1$ such that $p_s(\sigma)<p<p_L(\sigma)$, which implies $m<m_s$ by \eqref{interm1}. We have proved in Lemma \ref{lem.connect} that the orbit stemming from $P_0$ in the phase space associated to the system \eqref{PSsyst} but contained in the invariant plane $\{X=0\}$ enters the stable focus or node $P_2$. Thus, by standard continuity arguments, there exists a tubular neighborhood in the phase space such that the orbits belonging to the two-dimensional unstable manifold of $P_0$ and contained in that neighborhood will also enter $P_2$. The local analysis of the profiles contained in such orbits done in Lemmas \ref{lem.P0} and \ref{lem.P2} completes the proof.

In order to prove \textbf{Part 1} of Theorem \ref{th.forward}, we perform a \emph{backward shooting} on the two-dimensional manifold entering $P_1$. First of all, let us fix $p\in(p_s(\sigma),p_L(\sigma))$ and notice that the orbits on the stable manifold of the critical point $P_1$ are characterized by
\begin{equation}\label{var.P1}
Y(\eta)\to-\frac{N-2}{m}, \qquad Z(\eta)\sim CX(\eta)^{(N-2)(p-p_c(\sigma))/N(m_c-m)}, \qquad C\in(0,\infty),
\end{equation}
as $\eta\to\infty$, where the second approximation in \eqref{var.P1} follows readily from an integration of the first order terms between the first and third equations of the system \eqref{PSsyst}. We thus observe that the limits of this manifold are, on the one hand, the orbit entering $P_1$ contained in the invariant plane $\{X=0\}$, corresponding to the integration constant $C=\infty$ in \eqref{var.P1} and on the other hand, the orbit entering $P_1$ contained in the invariant plane $\{Z=0\}$, corresponding to the integration constant $C=0$ in \eqref{var.P1}. In the former, we understand the limit $C=\infty$ in the sense of rewriting the equivalence in \eqref{var.P1} as 
$$
X(\eta)\sim\left[\frac{1}{C}Z(\eta)\right]^{N(m_c-m)/(N-2)(p-p_c(\sigma))}, \qquad {\rm as} \ \eta\to\infty
$$
and letting then $C_1=1/C\to0$ as a limiting case. On the one hand, we have proved in Lemma \ref{lem.connect} that for $p>p_s(\sigma)$ the orbit entering $P_1$ and completely contained in the invariant plane $\{X=0\}$ comes from the unstable node $Q_2$, hence, by continuity, there exists $C^{*}>0$ such that, for any $C\in(C^*,\infty)$, the orbit in \eqref{var.P1} corresponding to such constant $C$ arrives from $Q_2$. Let us next prove that the unique orbit entering the saddle point $P_1$ and contained in the invariant plane $\{Z=0\}$ arrives from the unstable node $Q_1$. The reduced system in the plane $\{Z=0\}$ writes
\begin{equation}\label{syst.z0}
\left\{\begin{array}{ll}\dot{X}=X(2+(1-m)Y), \\ \dot{Y}=X-(N-2)Y-mY^2+\frac{p-m}{\sigma+2}XY\end{array}\right.
\end{equation}
and we notice that the flow of the system \eqref{syst.z0} on the line $\{Y=-(N-2)/m\}$ has negative sign, while the flow of the system \eqref{syst.z0} on the line $\{Y=-2/(1-m)\}$ has positive sign. Since the orbit entering $P_1$ is tangent to the eigenvector 
$$
e_1=\left(\frac{(N-2m-2)(\sigma+2)}{(N-2)(p-p_c(\sigma))},1\right),
$$
and, as $m<m_c$,
$$
N-2m-2>N-2-\frac{2(N-2)}{N}=(N-2)\left(1-\frac{2}{N}\right)=\frac{(N-2)^2}{N}>0,
$$
it follows that the orbit enters $P_1$ from inside the half-plane $\{Y>-(N-2)/m\}$. Coupling this information with the direction of the flow, we infer that the orbit entering $P_1$ in the invariant plane $\{Z=0\}$ lies forever inside the strip $\{-(N-2)/m<Y<-2/(1-m)\}$. But inside this strip, we have $\dot{X}<0$, thus it is easy to prove that this orbit starts from a critical point with $X=\infty$ and $Y/X\to0$ (as $Y$ is uniformly bounded), and this is $Q_1$. Since $Q_1$ behaves as an unstable node, as proved in Lemma \ref{lem.Q1}, we find by continuity that there exists $C_*>0$ such that the orbit entering $P_1$ and corresponding to any constant $C\in(0,C_*)$ in \eqref{var.P1} starts from $Q_1$. 

Let us denote for simplicity by $l_C$ the orbit corresponding to the parameter $C\in(0,\infty)$ in \eqref{var.P1}. We can thus define the following three sets 
\begin{equation*}
\begin{split}
&\mathcal{A}=\{C\in(0,\infty): {\rm the \ orbit} \ l_C \ {\rm comes \ from} \ Q_1\},\\
&\mathcal{C}=\{C\in(0,\infty): {\rm the \ orbit} \ l_C \ {\rm comes \ from} \ Q_2\},\\
&\mathcal{B}=(0,\infty)\setminus(\mathcal{A}\cup\mathcal{C}).
\end{split}
\end{equation*}
The previous considerations show that $\mathcal{A}$ is an open set containing an interval $(0,C_*)$ and $\mathcal{C}$ is an open set containing an interval $(C^*,\infty)$. It follows by standard topological arguments that $\mathcal{B}$ is a non-empty set. Let $C\in\mathcal{B}$. We show next that the orbit $l_C$ comes from the critical point $P_0$, ending the proof of this part. 

Since there are no more critical points except for $P_0$ allowing for unstable manifold, it suffices to prove that any $\alpha$-limit set is a critical point. Assume for contradiction that there exists an $\alpha$-limit set for an orbit entering $P_1$ and which is different from $P_0$. We deduce from \cite[Theorem 1, Section 3.2]{Pe} that the $\alpha$-limit set is a compact in the phase space associated to the system \eqref{PSsyst}, thus, there exist non-negative constants $K_1$, $K_2$ such that $0\leq K_1\leq X\leq K_2$ on the $\alpha$-limit set, whence 
\begin{equation}\label{interm28}
\left(\frac{mK_1}{\alpha}\right)^{1/(1-m)}\xi^{-2/(1-m)}\leq f(\xi)\leq \left(\frac{mK_2}{\alpha}\right)^{1/(1-m)}\xi^{-2/(1-m)}.
\end{equation}
This implies that in terms of profiles, the $\alpha$-limit set is attained as $\xi\to0$ by standard results on maximal intervals for solutions to differential equations. We then set 
$$
g(\xi):=\xi^{2/(1-m)}f(\xi) 
$$
and obtain after straightforward calculations that $g$ solves the differential equation 
\begin{equation}\label{ODE.alpha}
\begin{split}
\xi^{2}(g^m)''(\xi)&+\left(N-1-\frac{4m}{1-m}\right)\xi(g^m)'(\xi)+\frac{2m(mN-N+2)}{(1-m)^2}g^m(\xi)\\
&-\frac{1}{1-m}g(\xi)-\beta\xi g'(\xi)+\xi^{-L/(1-m)}g^p(\xi)=0,
\end{split}
\end{equation}
where we recall that $L<0$ is defined in \eqref{const.L}. Let us now take $\xi_0\in(0,\infty)$ to be a local maximum point of $g$. By evaluating \eqref{ODE.alpha} at $\xi=\xi_0$ and taking into account that $g'(\xi_0)=(g^m)'(\xi_0)=0$ and $(g^m)''(\xi_0)\leq0$, we find 
\begin{equation}\label{interm29}
\xi_0^{-L/(1-m)}g^p(\xi_0)\geq\frac{1}{1-m}g(\xi_0)-\frac{2m(mN-N+2)}{(1-m)^2}g^m(\xi_0)>\frac{1}{1-m}g(\xi_0),
\end{equation}
since $mN-N+2<0$ as $m<m_c$. We further infer from \eqref{interm28}, \eqref{interm29} and the definition of $g$ that
$$
\left(\frac{mK_2}{\alpha}\right)^{1/(1-m)}\geq g(\xi_0)>\left(\frac{1}{1-m}\right)^{1/(p-1)}\xi_0^{L/(p-1)(1-m)},
$$
which, together with the fact that $p>1>m$ and $L<0$, gives 
$$
\xi_0\geq K(m,p,\sigma):=\left[\left(\frac{mK_2}{\alpha}\right)^{1/(1-m)}(1-m)^{1/(p-1)}\right]^{(p-1)(1-m)/L}.
$$
It follows that the function $g$ cannot oscillate for $\xi\in(0,K(m,p,\sigma))$, having at most two intervals of monotonicity. This implies in particular that there exists $\lim\limits_{\xi\to0}g(\xi)$ and thus there also exists the limit 
$$
l:=\lim\limits_{\xi\to0}X(\xi)=\frac{\alpha}{m}\lim\limits_{\xi\to0}g(\xi)^{1-m}.
$$
If $l=0$, then the $\alpha$-limit set is contained in the invariant plane $\{X=0\}$ and, by the Poincar\'e-Bendixon theory and Lemma \ref{lem.explicit}, it must be a critical point, since there are no limit cycles for $p\neq p_s(\sigma)$. If $l>0$, it follows that there exists an $\alpha$-limit contained in the plane $\{X=l\}$. The first equation of the system \eqref{PSsyst} then leads to $l(2+(1-m)Y)=0$, thus $Y=-2/(1-m)$ on this $\alpha$-limit, hence $\dot{Y}=0$, which also gives a fixed value for the $Z$ component as follows from the second equation of the system \eqref{PSsyst}. This shows that the $\alpha$-limit in any case reduces to a critical point, which can only lie in the plane $\{X=0\}$, and the only point different from $Q_2$ allowing for a non-trivial unstable manifold is $P_0$. This argument completes the proof of Part 1.

We are left only with \textbf{Part 3}, the non-existence result. Notice first that this is immediate if $p<p_c(\sigma)$, since the critical point $P_2$ no longer exists and the analysis performed in Lemma \ref{lem.P1} gives that there are no non-trivial orbits entering the point $P_1$ if either $p\leq p_c(\sigma)$ or $m>m_c$ (which, together with \eqref{interm1}, authomatically implies $p<p_c(\sigma)$ in our range of parameters). Let us even observe that, if $m>m_c$, then $P_1$ becomes an unstable node. The same is valid in dimensions $N\in\{1,2\}$, when $m_c\leq0<m$, with the unessential difference that a transcritical bifurcation occurs in dimension $N=2$ as $P_1$ and $P_0$ coincide. In any of these cases, the only critical points allowing for a stable manifold are the stable node $Q_3$ and (if $m+p\leq2$) the critical point $Q_4$ analyzed in Lemmas \ref{lem.Q4a} and \ref{lem.Q4b}. All these local behaviors do not give rise to solutions to Eq. \eqref{eq1} since in one case, the contact condition $(f^m)'(\xi_0)$ is not fulfilled at the point $\xi_0\in(0,\infty)$ where $f(\xi_0)=0$, while the other two are vertical asymptotes.

We thus only have to work a bit more to prove \textbf{non-existence for $p_c(\sigma)\leq p\leq p_s(\sigma)$}. The most important step is to prove that there are no orbits connecting $P_0$ and $P_1$ in this range of $p$. Fix thus $p\in[p_c(\sigma),p_s(\sigma)]$ and consider the cylinder of equation \eqref{cylinder} in the phase space associated to the system \eqref{PSsyst}. On the one hand, we have proved that, for such values of $p$, the orbit going out of $P_0$ and fully contained in the invariant plane $\{X=0\}$ stays in the exterior region to the cylinder (except for $p=p_s$ when it is the base of the cylinder and connects to $P_1$ but no non-trivial profiles are contained in it). Moreover, all the other orbits in the unstable manifold of $P_0$ begin into the positive half-space $\{Y>0\}$, as shown in Lemma \ref{lem.P0}, thus trivially in the exterior region to the cylinder \eqref{cylinder}. On the other hand, we want to show that the orbits on the stable manifold of the critical point $P_1$ reach this point through the interior of the cylinder \eqref{cylinder}. But this is rather easy in this case, since the Hartman-Grobman Theorem ensures that these orbits enter $P_1$ tangent to the directions spanned by the two eigenvectors $e_1$ and $e_3$ given in Lemma \ref{lem.P1}, and both of them are pointing towards the region interior to the cylinder for $p\in[p_c(\sigma),p_s(\sigma)]$ (with the exception of $p=p_s(\sigma)$ when one orbit is exactly the contour of the cylinder, as established in Lemma \ref{lem.explicit}, but the other vector still points out to the interior of the cylinder and the conclusion does not change), thus any direction arriving to $P_1$ from the quadrant $\{X\geq0, Z\geq0\}$ enters through the interior of the cylinder. Moreover, the flow of the system \eqref{PSsyst} on the cylinder \eqref{cylinder}, considered with normal vector  
\begin{equation}\label{normal}
\overline{n}(X,Y,Z)=\left(0,\frac{N+\sigma}{N-2}(2mY+N-2),1\right)
\end{equation}
is given by the sign of the following expression obtained as the scalar product between $\overline{n}(X,Y,Z)$ and the vector field of the system \eqref{PSsyst}
\begin{equation}\label{flow.cyl}
F(X,Y)=\frac{N+\sigma}{N-2}\left[(p_s(\sigma)-p)Y^2(mY+N-2)+X\left(1+\frac{p-m}{\sigma+2}Y\right)(2mY+N-2)\right].
\end{equation}
Notice first that, if $p\leq p_s(\sigma)$, we have 
$$
\frac{\sigma+2}{p-m}-\frac{N-2}{2m}=\frac{(N-2)(p_s(\sigma)-p)}{2m(p-m)}\geq0,
$$
hence the expression in \eqref{flow.cyl} is obviously positive except, possibly, for the interval where the second term in the expression of $F(X,Y)$ is negative, that is,
\begin{equation}\label{interval}
-\frac{\sigma+2}{p-m}<Y<-\frac{N-2}{2m}.
\end{equation}
Thus, if any of the orbits on the unstable manifold of the point $P_0$ connects to $P_1$, it should cross the surface of the cylinder \eqref{cylinder} from its exterior into its interior region, and that can be done only in the region given by \eqref{interval}. However, in that region, the coordinate $Z$ is still increasing along these trajectories, since $\dot{Z}>0$ for $Y>-(\sigma+2)/(p-m)$, while the coordinate $Z$ on the cylinder attains its maximum at $Y=-(N-2)/2m$ and then decreases for $Y$ as in \eqref{interval}. Since the orbits on the unstable manifold of $P_0$ lie still ``above" the surface of the cylinder (in the sense of a higher $Z$ for the same $Y$), they will still remain ``above" afterwards as their $Z$ still increases while $Y$ is as in \eqref{interval} and thus they can never cross the surface of the cylinder. Since the trajectories on the stable manifold of $P_1$ enter through the interior of the cylinder, the points $P_0$ and $P_1$ cannot connect if $p\leq p_s(\sigma)$ (except for the explicit orbit \eqref{cylinder} for $p=p_s(\sigma)$). Moreover, in the same interval for $p$, there exists a unique orbit entering the critical point $P_2$, which corresponds to the explicit, stationary solution
\begin{equation}\label{stat.sol}
u(x)=K|x|^{-(\sigma+2)/(p-m)}, \qquad K=\left[\frac{m(\sigma+2)(N-2)(p-p_c(\sigma))}{(p-m)^2}\right]^{1/(p-m)},
\end{equation} 
which obviously connects $Q_1$ to $P_2$. Thus, there is no connection between $P_0$ and either $P_1$ or $P_2$ and the proof is complete.
\end{proof}
We plot in Figure \ref{fig2} some orbits going out of $P_0$ as the result of some numerical experiments which show how, if $p<p_s(\sigma)$, the orbits connect to the stable node $Q_3$, while if $p>p_s(\sigma)$, there are orbits entering the stable node $Q_3$ and also orbits entering the stable point $P_2$. Although the proof of Theorem \ref{th.forward} uses a different shooting, the plots suggest how an orbit connecting $P_0$ and $P_1$ appears for $p>p_s(\sigma)$ as a separatrix on the two-dimensional unstable manifold of $P_0$ between orbits going to $Q_3$ and orbits entering $P_2$.

\begin{figure}[ht!]
  \begin{center}
  \subfigure[$p<p_s(\sigma)$]{\includegraphics[width=7.5cm,height=6cm]{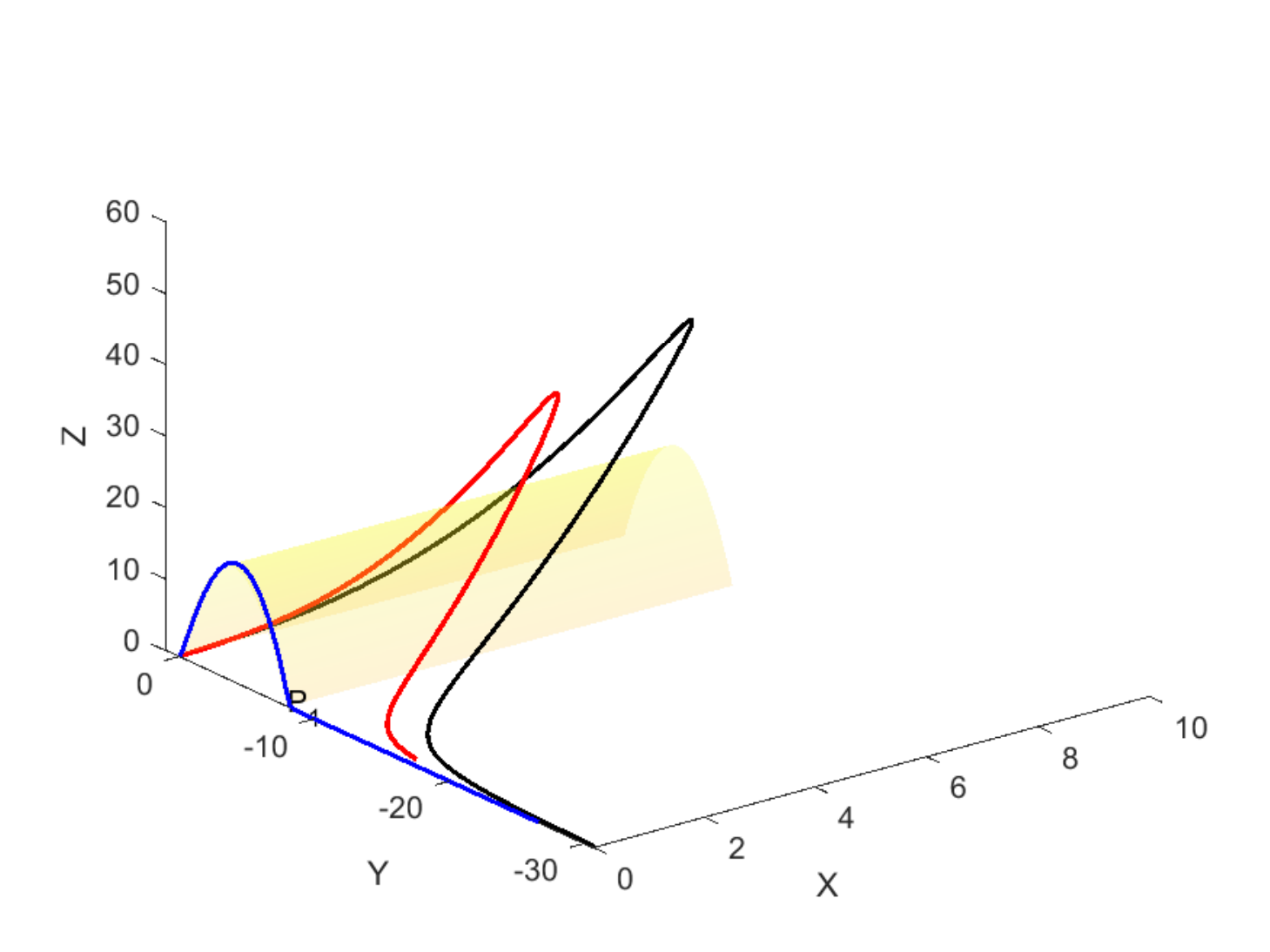}}
  \subfigure[$p>p_s(\sigma)$]{\includegraphics[width=7.5cm,height=6cm]{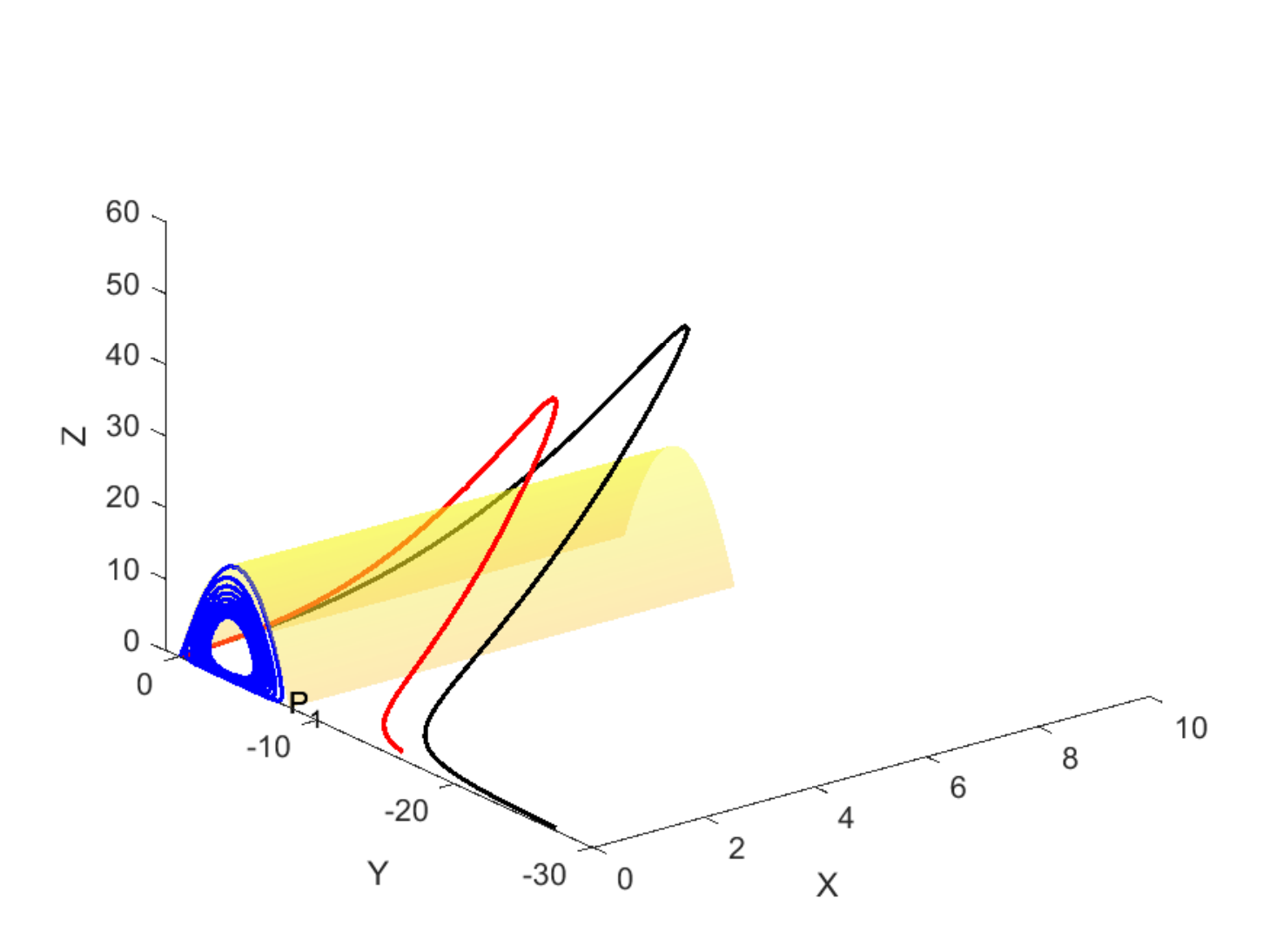}}
  \end{center}
  \caption{Orbits from $P_0$ for $p<p_s(\sigma)$ and for $p>p_s(\sigma)$ containing profiles for global solutions. Experiments for $m=0.25$, $N=4$, $\sigma=4$ and $p=1.74$, respectively $p=1.8$, where $p_s(\sigma)=1.75$.}\label{fig2}
\end{figure}

\section{Finite time extinction profiles. The phase space}\label{sec.ext.local}

This and the next section are devoted to the study of the self-similar solutions presenting finite time extinction, in the form \eqref{backward} and whose profiles $f(\xi)$ solve the differential equation \eqref{ODE.extinction}. In this section we give the preliminaries needed for the proof of Theorem \ref{th.extinction}, and the actual proof will be completed in the forthcoming Section \ref{sec.extinction}. Many technical details will be similar to the analysis already done in the previous chapters, thus we will only give the new arguments. Starting from \eqref{ODE.extinction}, we perform the same change of variables \eqref{PSchange} to get the system
\begin{equation}\label{PSsyst.ext}
\left\{\begin{array}{ll}\dot{X}=X(2+(1-m)Y), \\ \dot{Y}=-X-(N-2)Y-Z-mY^2-\frac{p-m}{\sigma+2}XY, \\ \dot{Z}=Z(\sigma+2+(p-m)Y),\end{array}\right.
\end{equation}
where again $X\geq0$, $Z\geq0$. The finite critical points are the same ones $P_0$, $P_1$, $P_2$ as for the system \eqref{PSsyst}, together with a new critical point 
$$
P_3=\left(\frac{2(\sigma+2)(mN-N+2)}{L(1-m)},-\frac{2}{1-m},0\right),
$$
which in our range (that is $L<0$) exists only if $mN-N+2\leq0$, hence for $m\leq m_c$. Since the invariant plane $\{X=0\}$ is the same one for the system \eqref{PSsyst.ext} as for the initial system \eqref{PSsyst}, the local analysis of the critical points $P_0$, $P_1$ and $P_2$ is practically the same as the one given in Lemmas \ref{lem.P0}, \ref{lem.P1} and \ref{lem.P2}, with the only minor differences listed below: 

$\bullet$ The eigenvectors of the linearization of the system \eqref{PSsyst.ext} in a neighborhood of $P_0$ are 
$$
e_1=(1,-N,0), \ \ e_2=(0,1,0), \ \ e_3=(0,-1,N+\sigma).
$$
In the local analysis of the profiles on the unstable manifold of $P_0$, we obtain
$$
\frac{dY}{dX}\sim-\frac{(N-2)Y+X+Z}{2X}=-\frac{(N-2)Y+X+CX^{(\sigma+2)/2}}{2X},
$$
which gives by integration that
\begin{equation}\label{interm3.ext}
Y(\eta)\sim-\frac{X(\eta)}{N}-\frac{C}{N+\sigma}X(\eta)^{(\sigma+2)/2}+C_1X(\eta)^{-(N-2)/2}, \qquad {\rm as} \ \eta\to-\infty.
\end{equation}
Similar arguments as in the proof of Lemma \ref{lem.P1} then lead to the local behavior \eqref{beh.P0b}.

$\bullet$ In the local analysis of $P_1$, the eigenvector $e_1$ changes as follows 
$$
e_1=\left(\frac{(N-2m-2)(\sigma+2)}{(N-2)(p-p_c(\sigma))},-1,0\right).
$$

The main novelty is the critical point $P_3$, which will be very important in the sequel. The local analysis in a neighborhood of it is given below.
\begin{lemma}\label{lem.P3}
The critical point $P_3$ is a saddle point whose restriction to the invariant plane $\{Z=0\}$ is either a stable node or a stable focus, while there is a unique orbit going out of it into the half-space $\{Z>0\}$. This orbit translates into profiles with a vertical asymptote at the origin 
\begin{equation}\label{beh.P3}
f(\xi)\sim C_0\xi^{-2/(1-m)}, \qquad C_0=\left[\frac{2m(N-2)(m_c-m)}{1-m}\right]^{1/(1-m)}. 
\end{equation}
\end{lemma}
\begin{proof}
The linearization of the system \eqref{PSsyst.ext} in a neighborhood of $P_3$ has the matrix 
$$
M(P_3)=\left(
         \begin{array}{ccc}
           0 & \frac{2(\sigma+2)(mN-N+2)}{L} & 0 \\[1mm]
           \frac{L}{(1-m)(\sigma+2)} & -(N-2)+\frac{4m}{1-m}-\frac{2(p-m)(mN-N+2)}{L(1-m)} & -1 \\[1mm]
           0 & 0 & -\frac{L}{1-m} \\
         \end{array}
       \right),
$$
whose eigenvalues satisfy 
\begin{equation*}
\begin{split}
&\lambda_3=-\frac{L}{1-m}>0, \qquad \lambda_1\lambda_2=-\frac{2(mN-N+2)}{1-m}>0, \\
&\lambda_1+\lambda_2=\frac{(1-m)^2(\sigma+2)N+2(m^2-1)\sigma+4(mp-1)}{L(1-m)}<0,
\end{split}
\end{equation*}
the latter inequality being true since $L<0$ and the numerator of the expression of $\lambda_1+\lambda_2$ is an increasing function of $p$ whose value at $p=p_c(\sigma)$ is equal to $(mN-N+2)^2(\sigma+2)/(N-2)>0$. Thus, $\lambda_1$ and $\lambda_2$ are either both negative real numbers or conjugate complex numbers with negative real parts, thus we obtain a stable node or focus whose basin of attraction is fully included in the invariant plane $\{Z=0\}$ and a single orbit going out from $P_3$ into the region $\{Z>0\}$. This unstable orbit contains profiles for which 
$$
\lim\limits_{\eta\to-\infty}X(\eta)=\frac{2(\sigma+2)(mN-N+2)}{L(1-m)},
$$
which leads to \eqref{beh.P3} after undoing the change of variable in \eqref{PSchange}.
\end{proof}

\noindent \textbf{Remark}. The flow of the system \eqref{PSsyst.ext} on the plane $\{Y=0\}$ has negative direction. Thus, any orbit entering the half-space $\{Y\leq0\}$ remains there forever. In particular, all the trajectories of the unstable manifold of $P_0$ are contained in the half-space $\{Y\geq0\}$ and thus the profiles with local behavior \eqref{beh.P0b} are decreasing for $\xi>0$. 

We pass now to the study of the critical points of the system \eqref{PSsyst.ext} at infinity. While it is easy to check that the local analysis of the points $Q_2$ and $Q_3$ remains the same as in Lemma \ref{lem.Q23} (some terms change sign in the system \eqref{PSinf2} but not the linear terms, which are the only ones counting in the analysis), the system \eqref{PSinf1} becomes in the same variables $(x,y,z)$ and $\eta_1$ as in \eqref{change2}-\eqref{indep2}
\begin{equation}\label{PSinf1.ext}
\left\{\begin{array}{ll}\dot{x}=x[(m-1)y-2x],\\
\dot{y}=-y^2-\frac{p-m}{\sigma+2}y-x-Nxy-xz,\\
\dot{z}=z[(p-1)y+\sigma x],\end{array}\right.
\end{equation}
with critical points $Q_1=(0,0,0)$, $Q_5=(0,-\beta/\alpha,0)$ and $Q_{\gamma}=(0,0,\gamma)$. The analysis of the point $Q_1$ is similar to the one performed in Lemma \ref{lem.Q1} with the difference that the single nonzero eigenvalue is now $-(p-m)/(\sigma+2)<0$, thus we have a saddle point instead of an unstable node, with a two-dimensional center manifold with unstable direction of the flow and whose orbits contain profiles with the vertical asymptote \eqref{beh.Q1}. The analysis of the critical point $Q_5$ is totally similar as the one in Lemma \ref{lem.Q5}, with a two-dimensional unstable manifold fully contained in the plane $\{Z=0\}$ and a one-dimensional stable manifold included in the plane $\{X=0\}$. Passing now to the analysis of the critical point $Q_4$, we notice that a statement which is analogous to Lemma \ref{lem.Q4} holds true, that is,
\begin{lemma}\label{lem.Q4.ext}
There are no solutions to Eq. \eqref{ODE.extinction} satisfying simultaneously the limits in \eqref{limits} either as $\xi\to\xi_0$, $\xi<\xi_0$ for some $\xi_0\in(0,\infty)$ or as $\xi\to\infty$.
\end{lemma}
A sketch of the proof is given at the end of the Appendix, along the lines of the proof of Lemma \ref{lem.Q4}. Thus, if we apply the change of variable $w:=xz$ in the system \eqref{PSinf1.ext}, we get the new system
\begin{equation}\label{PSsyst3.ext}
\left\{\begin{array}{lll}\dot{x}=x[(m-1)y-2x],\\[1mm]
\dot{y}=-y^2-\frac{p-m}{\sigma+2}y-x-Nxy-w,\\[1mm]
\dot{w}=w[(\sigma-2)x+(m+p-2)y],\end{array}\right.
\end{equation}
and we again obtain the two critical points $Q_1'=(0,0,0)$ and $Q_5'=(0,-(p-m)/(\sigma+2),0)$, Lemma \ref{lem.Q4.ext} implies that there are no orbits entering the critical point $Q_4$ in the system \eqref{PSsyst.ext} other than, possibly, the ones contained in the critical points $Q_1'$ and $Q_5'$ in the system \eqref{PSsyst3.ext}. For the analysis of $Q_1'$, we proceed similarly as in the proof of Lemma \ref{lem.Q4b} by introducing the new variable
$$
t=\frac{p-m}{\sigma+2}y+x+w,
$$
and in this case the reduced system giving the flow on the center manifold is
\begin{equation}\label{syst.red.ext}
\left\{\begin{array}{ll}\dot{x}=\frac{1}{\beta}x[x+\alpha(1-m)w],\\[1mm]
\dot{w}=\frac{1}{\beta}w[2x-\alpha(m+p-2)w].\end{array}\right.
\end{equation}
It thus follows that there can only be unstable and saddle sectors, since $\dot{x}>0$ in the system \eqref{syst.red.ext}. We also notice that there exists an explicit orbit $x=\alpha(p-1)w$, which corresponds to a unique profile with the vertical asymptote 
\begin{equation}\label{beh.Q4b.ext}
f(\xi)\sim \left(\frac{1}{p-1}\right)^{1/(p-1)}\xi^{-\sigma/(p-1)}, \qquad {\rm as} \ \xi\to0.
\end{equation}
The local analysis of the critical point $Q_5'$ can be done along the same lines as in the proof of Lemma \ref{lem.Q4a}. There are new profiles only if $m+p<2$, when $Q_5'$ becomes an unstable node and the profiles contained in such orbits have the vertical asymptote 
\begin{equation}\label{beh.Q4a.ext}
f(\xi)\sim\left[\frac{\beta(1-m)}{2m}\xi^2-K\right]^{-1/(1-m)}, \qquad {\rm as} \ \xi\to\xi_0=\sqrt{\frac{2mK}{\beta(1-m)}}\in(0,\infty).
\end{equation}
We omit the details, since these two critical points will not play any role in the rest of the analysis.

\section{Proof of Theorem \ref{th.extinction}}\label{sec.extinction}

The two-dimensional unstable manifold of $P_0$ can be described as the following one-parameter family of orbits
\begin{equation}\label{orbits.Q1}
l_K: Z=KX^{(\sigma+2)/2}, \qquad Y=-\frac{X}{N}+o(X), \qquad K\in[0,\infty],
\end{equation}
with the convention that the orbit $l_0$ is contained in the invariant plane $\{Z=0\}$ and the orbit $l_{\infty}$ in the invariant plane $\{X=0\}$. We classify the orbits $l_K$ in \eqref{orbits.Q1} in the following three sets
\begin{equation*}
\begin{split}
&\mathcal{U}=\{K\in(0,\infty): \dot{Y}(\eta)\leq0 \ {\rm for \ any} \ \eta\in\real, \ Y(\eta)\to-\infty \ {\rm as} \ \eta\to\infty\},\\
&\mathcal{V}=\{K\in(0,\infty): {\rm there \ exists} \ \eta_0\in\real, \ \dot{Y}(\eta_0)>0\},\\
&\mathcal{W}=(0,\infty)\setminus(\mathcal{U}\cup\mathcal{V}).
\end{split}
\end{equation*}
Notice that the splitting into three sets this time is not with respect to the endpoint of the trajectories (as in fact, there are only two points where they can end, that is, the stable node $Q_3$ and the point $P_1$) but instead, with respect to the monotonicity of the coordinate $Y$ along these trajectories. It is obvious from the definition that the set $\mathcal{V}$ is open. 
\begin{lemma}\label{lem.ext.1}
If $K\in\mathcal{U}$, then the orbit $l_K$ enters the stable node $Q_3$. In particular, $\mathcal{U}$ is open.
\end{lemma}
\begin{proof}
We know that $Y$ is non-increasing along the trajectory $l_K$ with $K\in\mathcal{U}$ and that $Y(\eta)\to-\infty$. Hence, there exists $\eta_0\in\real$ such that 
$$
Y(\eta)<\min\left\{-\frac{\sigma+2}{p-m},-\frac{2}{1-m}\right\}, \qquad {\rm for} \ \eta\in(\eta_0,\infty).
$$
Thus, $\dot{X}(\eta)<0$ and $\dot{Z}(\eta)<0$ for $\eta\in(\eta_0,\infty)$ and in particular, there exist $X_0:=\lim\limits_{\eta\to\infty}X(\eta)\geq0$ and $Z_0:=\lim\limits_{\eta\to\infty}Z(\eta)\geq0$ on the trajectory $l_K$. It follows that the orbit $l_K$ ends up at a critical point with $Y\to-\infty$, $X=X_0$, $Z=Z_0$, that is, $X/Y\to0$ and $Z/Y\to0$ as $\eta\to\infty$. The analysis done in Section \ref{sec.inf} proves that this point is $Q_3$. Since $Q_3$ is a stable node, standard continuity arguments imply that $\mathcal{U}$ is an open set.
\end{proof}
Let us characterize now the set $\mathcal{W}$, which is our goal. We first remark that, if $K\in\mathcal{W}$, then $\dot{Y}(\eta)\leq0$ on the trajectory $l_K$ and there exists 
$$
Y_0:=\lim\limits_{\eta\to\infty}Y(\eta)\in(-\infty,0).
$$
\begin{lemma}\label{lem.ext.2}
If $K\in\mathcal{W}$ and $p_c(\sigma)<p<p_s(\sigma)$, then $Y_0=-(N-2)/m$ and the orbit $l_K$ connects to $P_1$.
\end{lemma}
\begin{proof}
The $\omega$-limit set of the trajectory $l_K$ has to belong to the plane $\{Y=Y_0\}$. Moreover, since 
$$
\frac{\sigma+2}{p-m}-\frac{2}{1-m}=-\frac{L}{(p-m)(1-m)}>0,
$$
we obtain that at least one of the components $X(\eta)$, $Z(\eta)$ of the trajectory $l_K$ is monotone in a left neighborhood of the plane $\{Y=Y_0\}$, since $Y_0$ cannot take at the same time the values $-(\sigma+2)/(p-m)$ and $-2/(1-m)$ where the monotonicity of $Z(\eta)$, respectively $X(\eta)$ changes. In fact, if $Y_0$ is not equal to one of these two values, then both $X(\eta)$ and $Z(\eta)$ are monotone for $\eta\in(\eta_0,\infty)$ for some $\eta_0$ sufficiently large. Standard arguments give that the $\omega$-limit is a critical point. Since $p\in(p_c(\sigma),p_s(\sigma))$, this critical point cannot be $P_2$, which is for this range of $p$ a saddle with a unique orbit entering it, which is the stationary solution given in \eqref{stat.sol}. Lemma \ref{lem.P3} shows that the critical point also cannot be $P_3$ if $K>0$, as in the definition of $\mathcal{W}$. It thus follows that $l_K$ connects to $P_1$, as claimed.
\end{proof}
The following result shows that $\mathcal{U}$ is nonempty. 
\begin{lemma}\label{lem.ext.3}
Let $p\in(p_c(\sigma),p_s(\sigma))$. There exists $K_1>0$ such that $(K_1,\infty)\subseteq\mathcal{U}$.
\end{lemma}
\begin{proof}
Since the reduced system in the invariant plane $\{X=0\}$ is the same one for \eqref{PSsyst.ext} as for \eqref{PSsyst}, we have already proved in Lemma \ref{lem.connect} that the orbit $l_{\infty}$ (in our notation) ends up by entering the stable node $Q_3$. It remains to show that $\dot{Y}(\eta)\leq0$ for any $\eta\in\real$ on this trajectory. Assume for contradiction that this is not true, thus, there exists a first $\eta_0\in\real$ such that $\dot{Y}(\eta_0)=0$ and $\ddot{Y}(\eta_0)\geq0$. The former condition gives 
\begin{equation}\label{interm30}
Z(\eta_0)=-(N-2)Y(\eta_0)-mY(\eta_0)^2,
\end{equation} 
while the latter condition, together with \eqref{interm30}, gives
\begin{equation*}
\begin{split}
\ddot{Y}(\eta_0)&=-(N-2)\dot{Y}(\eta_0)-2mY(\eta_0)\dot{Y}(\eta_0)-\dot{Z}(\eta_0)=-\dot{Z}(\eta_0)\\
&=-Z(\eta_0)[\sigma+2+(p-m)Y(\eta_0)]=Y(\eta_0)[N-2+mY(\eta_0)][\sigma+2+(p-m)Y(\eta_0)],
\end{split}
\end{equation*}
and, taking into account that $Y(\eta_0)<0$, $\ddot{Y}(\eta_0)$ is non-negative if and only if $Y(\eta_0)\in[-(N-2)/m,-(\sigma+2)/(p-m)]$. But, since $l_{\infty}$ connects to $Q_3$, if such a change of monotonicity exists, then there should be at least another one in order for $Y(\eta)$ to become again decreasing. Thus, there exists $\eta_1>\eta_0$ such that $\dot{Y}(\eta_1)=0$ and $\ddot{Y}(\eta_1)<0$, which by the same calculation as above gives that 
$$
-\frac{N-2}{m}\leq Y(\eta_0)\leq-\frac{\sigma+2}{p-m}<Y(\eta_1)<0.
$$ 
In particular, the previous arguments imply that the trajectory $l_{\infty}$ must cross the vertical line $\{Y=-(\sigma+2)/(p-m)\}$ from left to right. But the flow of the system \eqref{syst.x0} over the line $\{Y=-(\sigma+2)/(p-m)\}$ is given by the sign of the expression 
$$
G(Z)=-Z+\frac{(N-2)(\sigma+2)(p-p_{c}(\sigma))}{(p-m)^2},
$$
which is positive only if $Z$ lies below the coordinate of the critical point $P_2$. Hence, this crossing can only take place through the closed region limited by the curve \eqref{cylinder} and the $Y$-axis. But this leads to a contradiction, since for $p\in(p_c(\sigma),p_s(\sigma))$, the flow on this curve has the positive direction, as shown in the proof of Lemma \ref{lem.connect} and thus the orbit $l_{\infty}$ cannot cross it to enter the mentioned closed region. This contradiction shows that $Y(\eta)$ is non-increasing over the trajectory $l_{\infty}$. By continuity with respect to the parameter on the orbits \eqref{orbits.Q1}, we infer that there is $K_1>0$ such that the same occurs for $K\in(K_1,\infty)$, hence $(K_1,\infty)\subseteq\mathcal{U}$.
\end{proof}
For the next proofs, it will be very useful to express better the two-dimensional manifold entering $P_1$. To this end, we translate this point to the origin by setting 
$$
U:=Y+\frac{N-2}{m}.
$$
In these new variables, the system \eqref{PSsyst.ext} writes
\begin{equation}\label{syst.SU}
\left\{\begin{array}{lll}\dot{X}=\frac{mN-N+2}{m}X+(1-m)XU,\\[1mm]
\dot{U}=\frac{(N-2)(p-p_c(\sigma))}{m(\sigma+2)}X+(N-2)U-Z-mU^2-\frac{p-m}{\sigma+2}XU,\\[1mm]
\dot{Z}=-\frac{(N-2)(p-p_c(\sigma))}{m}Z+(p-m)ZU,\end{array}\right.
\end{equation}
and we analyze the stable manifold in a neighborhood of the origin. On the one hand, by integration of the first order approximation of the equations for $\dot{X}$ and $\dot{Z}$, we find
\begin{equation}\label{interm31}
X\sim CZ^{-(mN-N+2)/(N-2)(p-p_c(\sigma))}=CZ^{(m_c-m)/(p-p_c(\sigma))}, \qquad C\geq0.
\end{equation}
On the other hand, fixing the constant $C>0$ and replacing $X$ by its first order approximation in \eqref{interm31}, we have
$$
\frac{dU}{dZ}\sim-m\frac{(N-2)U-Z+CZ^{(m_c-m)/(p-p_c(\sigma))}}{(N-2)(p-p_c(\sigma))Z},
$$
which by integration gives the first order approximation of the stable manifold entering $P_1$ as follows (where we undo the translation and come back to the initial variables of the system \eqref{PSsyst.ext})
\begin{equation}\label{var.P1.ext}
Y(\eta)\sim-\frac{N-2}{m}+\frac{mZ(\eta)}{(N-2)(p-p_c(\sigma))}-\frac{mC}{N-2m-2}Z(\eta)^{(m_c-m)/(p-p_c(\sigma))},
\end{equation}
as $\eta\to\infty$. With this expression, we can prove the following 
\begin{lemma}\label{lem.ext.4}
If $p\geq p_s(\sigma)$, all the orbits $l_K$ remain forever in the interior of the cylinder \eqref{cylinder} and do not connect to the point $P_1$. In particular, if $p=p_s(\sigma)$ we have $\mathcal{V}=(0,\infty)$.
\end{lemma}
\begin{proof}
Let $p\geq p_s(\sigma)$. Since the eigenvector $e_3$ points out to the interior of the cylinder \eqref{cylinder} and the other limiting orbit $l_{\infty}$ enters the interior of the cylinder as shown in \eqref{interm26}, it follows easily that in this range, all the two-dimensional unstable manifold of $P_0$ goes out into the interior region of the cylinder \eqref{cylinder}. Similarly to \eqref{flow.cyl}, the flow of the system \eqref{PSsyst.ext} over the surface \eqref{cylinder} with normal \eqref{normal} is given by the sign of the expression
\begin{equation}\label{flow.cyl.ext}
E(X,Y;p)=\frac{N+\sigma}{N-2}\left[(p_s(\sigma)-p)Y^2(mY+N-2)-X\left(1+\frac{p-m}{\sigma+2}Y\right)(2mY+N-2)\right],
\end{equation}
which in particular writes for $p=p_s(\sigma)$ as 
$$
E(X,Y;p_s(\sigma))=-(N+\sigma)X\left(1+\frac{2m}{N-2}Y\right)^2\leq0.
$$
It thus follows that the orbits $l_K$ cannot leave the interior of the cylinder if $p=p_s(\sigma)$. Moreover, if $p>p_s(\sigma)$, the expression $E(X,Y;p)$ might be positive only for
$$
-\frac{N-2}{2m}<Y<-\frac{\sigma+2}{p-m}.
$$
But in this interval, the component $Z(\eta)$ of any trajectory $l_K$ starts to decrease, as indicated by the third equation in \eqref{PSsyst.ext}, while the coordinate $Z$ on the cylinder still increases up to its maximum attained for $Y=-(N-2)/2m$. Hence, the orbits $l_K$ cannot leave the interior of the cylinder \eqref{cylinder} even in this interval of $Y$, thus they will remain forever inside the cylinder. On the contrary, let us notice that 
$$
\frac{m}{(N-2)(p-p_c(\sigma))}-\frac{1}{N+\sigma}\leq0
$$
for any $p\geq p_s(\sigma)$. From the previous estimate, together with the fact that in our range of interest $m<m_c$ we have $N-2m-2>0$, we deduce by inspecting the asymptotic expansion \eqref{var.P1.ext} that all the orbits enter $P_1$ through the exterior of the cylinder \eqref{cylinder} if $p\geq p_s(\sigma)$ (with the exception of the limiting orbit lying in the invariant plane $\{X=0\}$ for exactly $p=p_s(\sigma)$, which is the explicit curve \eqref{cylinder}). Hence there cannot be any connection between $P_0$ and $P_1$ if $p>p_s(\sigma)$ and also no other connection than the explicit curve given in Lemma \ref{lem.explicit} if $p=p_s(\sigma)$. Moreover, since we know that $P_2$ is a center in the invariant plane $\{X=0\}$ if $p=p_s(\sigma)$ and all the orbits $l_K$ with $K\in(0,\infty)$ stay inside the cylinder, we infer that $(0,\infty)\subseteq\mathcal{V}$ and the conclusion follows.
\end{proof}
The last preparatory result before going to the proof of Theorem \eqref{th.extinction} is the following
\begin{lemma}\label{lem.ext.5}
If $m\in(0,m_s)$, the orbit $l_0$ (inside the plane $\{Z=0\}$) connects $P_0$ to $P_3$.
\end{lemma}
\begin{proof}
The system \eqref{PSsyst.ext} reduces in the invariant plane $\{Z=0\}$ to the system 
\begin{equation}\label{syst.z0.ext}
\left\{\begin{array}{ll}\dot{X}=X(2+(1-m)Y), \\ \dot{Y}=-X-(N-2)Y-mY^2-\frac{p-m}{\sigma+2}XY.\end{array}\right.
\end{equation}
Since $p>p_c(\sigma)$, the flow of the system \eqref{syst.z0.ext} over the line $\{Y=-(\sigma+2)/(p-m)\}$ is given by the sign of the expression obtained by replacing this value of $Y$ into the second equation of \eqref{syst.z0.ext}, that is
$$
\frac{(\sigma+2)(N-2)(p-p_c(\sigma))}{(p-m)^2}>0,
$$
hence this line cannot be crossed from right to left. We thus find that the orbit $l_0$ stays forever in the strip $\{-(\sigma+2)/(p-m)<Y<0\}$ (since we recall that the flow over the line $\{Y=0\}$ has negative direction), to which also belongs the stable point $P_3$. Moreover, the system \eqref{syst.z0.ext} has no limit cycles, as the Dulac criterion \cite[Theorem 2, Section 3.9]{Pe} will imply. Indeed, if we multiply the vector field of the system \eqref{syst.z0.ext} by $X^{\eta}$ for some generic $\eta$ and compute the divergence of the obtained vector field, we find that
\begin{equation*}
\begin{split}
D:=\frac{\partial(X^{\eta}\dot{X})}{\partial X}+\frac{\partial(X^{\eta}\dot{Y})}{\partial Y}&=\left[(1-m)(\eta+1)-2m\right]X^{\eta}Y+\left[2(\eta+2)-(N-2)\right]X^{\eta}\\
&-\frac{p-m}{\sigma+2}X^{\eta+1}.
\end{split}
\end{equation*}
Choosing $\eta=(N-4)/2$, we cancel out the second term in brackets, and the divergence becomes 
$$
D=\frac{(N-2)(m_s-m)}{2}X^{\eta}Y-\frac{p-m}{\sigma+2}X^{\eta+1}<0,
$$
since we are dealing only with $Y$ negative. Bendixon's theory \cite[Section 3.9]{Pe} then implies that the orbit $l_0$ has to enter a critical point lying in the strip $\{-(\sigma+2)/(p-m)<Y<0\}$ and the only such point admitting a stable manifold is $P_3$.
\end{proof}
We are now ready to complete the proof of Theorem \ref{th.extinction}.
\begin{proof}[Proof of Theorem \ref{th.extinction}]
In order to prove \textbf{Part 1}, let us notice that the condition on $m$ given in the statement implies $1<p_s(\sigma)<p_L(\sigma)$. Since by Lemma \ref{lem.ext.4} we have $\mathcal{V}=(0,\infty)$ for $p=p_s(\sigma)$, it follows by continuity of the system \eqref{PSsyst.ext} with respect to the parameter $p$ that there exists $p_0(\sigma)\in(p_c(\sigma),p_s(\sigma))$ such that also $p_0(\sigma)>1$ and for any $p\in(p_0(\sigma),p_s(\sigma))$ we have $\mathcal{V}\neq\emptyset$. But we already know that for any $p\in(p_c(\sigma),p_s(\sigma))$ we have $\mathcal{U}\neq\emptyset$ as proved in Lemma \ref{lem.ext.3}. Since $\mathcal{U}$ is an open set by Lemma \ref{lem.ext.1} and $\mathcal{V}$ is an open set by definition, it follows by elementary topology that also $\mathcal{W}\neq\emptyset$. Lemma \ref{lem.ext.2} then ensures that for any $K\in\mathcal{W}$, the orbit $l_K$ connects $P_0$ and $P_1$. Translated in terms of profiles, it means, as it follows by the local analysis in Lemma \ref{lem.P1} and the beginning of Section \ref{sec.ext.local}, that there exists at least a profile with local behavior \eqref{beh.P0b} as $\xi\to0$ and \eqref{beh.P1} as $\xi\to\infty$, as claimed.

The proof of \textbf{Part 2} is then immediate. Let $p\in(p_s(\sigma),p_L(\sigma))$. It follows from Lemma \ref{lem.connect} that the orbit $l_{\infty}$ going out of $Q_1$ inside the invariant plane $\{X=0\}$ enters the stable point $P_2$. Since for $p>p_s(\sigma)$ the critical point $P_2$ is a stable point as given by Lemma \ref{lem.P2}, there exists $K(p)>0$, depending eventually on $p$, such that for any $K\in(K(p),\infty)$, the orbit $l_K$ connects $P_0$ to $P_2$. In terms of profiles, such orbit means a profile with local behavior \eqref{beh.P0b} as $\xi\to0$ and \eqref{beh.P2} as $\xi\to\infty$. Moreover, these are the only self-similar solutions for $p>p_s(\sigma)$: indeed, Lemma \ref{lem.ext.4} ensures that there are no orbits connecting $P_0$ to $P_1$ if $p>p_s(\sigma)$. 

In order to prove \textbf{Part 3}, the non-existence for $p<p_c(\sigma)$ and in particular for $m\geq m_c$, which is also implied authomatically in dimensions $N\in\{1,2\}$, follows readily in the same way as in the proof of Theorem \ref{th.forward}, by the non-existence of any tail behavior at the point $P_1$ and the mere non-existence of the point $P_2$. We have to work a bit more just to show that $p_0(\sigma)>p_c(\sigma)$, where $p_0(\sigma)$ is the exponent introduced in Part 1. That is, we want to show that in a right neighborhood of $p_c(\sigma)$, there are still no connections between $P_0$ and $P_1$. To this end, let us prove that all the orbits $l_K$ lie always in the region $\{0\leq X\leq X(P_3)\}$ if $Y<-2/(1-m)$. Since by Lemma \ref{lem.ext.5} the orbit $l_0$ connects $P_0$ with $P_3$ inside the plane $\{Z=0\}$, let us consider the cylinder built over this trajectory $l_0$, with $Z>0$ free variable. If $Y=\Phi(X)$ is a local chart of the curve $l_0$ (in the sense that, the curve itself might not be entirely the graph of a single function) and thus of the boundary of the cylinder, we infer from the first and second equations of the system \eqref{PSsyst.ext} that the flow of this system across this part of the cylinder is given by the sign of the expression obtained by the scalar product of the normal vector $\overline{n}=(-\Phi'(X),1,0)$ with the vector field of the system \eqref{PSsyst.ext}. Taking into account that $\Phi'(X)=dY/dX$ in the local chart we consider, this expression becomes
$$
-\Phi'(X)X(2+(1-m)\Phi(X))-X-(N-2)\Phi(X)-m\Phi(X)^2-\frac{p-m}{\sigma+2}X\Phi(X)-Z=-Z<0,
$$
thus the cylinder raised over the trajectory $l_0$ cannot be crossed from the interior to its exterior. We then complete the cylinder, starting from $Y=-2/(1-m)$, with the plane $\{X=X(P_3)\}$, on which the direction of the flow is also negative provided $-\infty<Y\leq-2/(1-m)$. Noticing, from \eqref{interm3.ext}, that the orbits $l_K$ with $K>0$ go out of $P_0$ more negatively than the orbit $l_0$ which is the basis of the cylinder, we infer that along any trajectory $l_K$ with $K>0$ we have $X(\eta)<X(P_3)$ if $Y(\eta)\leq-2/(1-m)$.

We further observe that, if $p=p_c(\sigma)$, then $(\sigma+2)/(p-m)=(N-2)/m$ and thus $\dot{Z}(\eta)>0$ if $Y(\eta)>-(N-2)/m$. Moreover, the flow of the system \eqref{PSsyst.ext} over the plane $\{Y=-(N-2)/m\}$ has always negative direction. It is then clear that all the orbits $l_K$ with $K\in(0,\infty)$, cross this plane and remain then forever in the region $\{Y<-(N-2)/m\}$. But in this region, we notice that $\dot{Z}<0$, $\dot{X}<0$ and
\begin{equation}\label{interm32}
\dot{Y}=-Z-\frac{(mY+N-2)(X+(N-2)Y)}{N-2},
\end{equation}
since $p=p_c(\sigma)$. By the previous analysis, we have $X\leq X(P_3)$, hence
$$
X+(N-2)Y\leq X(P_3)-\frac{(N-2)^2}{m}=\frac{2(N-2)}{1-m}-\frac{(N-2)^2}{m}=\frac{(N-2)(mN-N+2)}{m(1-m)}<0,
$$
since $mN-N+2<0$ if $m<m_c$. It then follows from \eqref{interm32} that also $\dot{Y}<0$ if $Y<-(N-2)/m$, hence if $p=p_c(\sigma)$, all the orbits $l_K$ with $K>0$ end up in the stable node $Q_3$. The same happens with the unique orbit going out of $P_3$, which goes out in the region $\{0<X<X(P_3)\}$, as it can be seen by inspecting the eigenvector corresponding to the only positive eigenvalue of the matrix $M(P_3)$ given in Lemma \ref{lem.P3}. It thus follows that the orbits $l_K$ with $K\in(0,\infty)$ together with their limits $l_{\infty}$ and the unique orbit going out of $P_3$ describe for $p=p_c(\sigma)$ a closed curve over the plane $\{Y=-(N-2)/m\}$ lying at strictly positive distance from the critical point $P_1$ and then connecting to the stable node $Q_3$. The continuity with respect to $p$ of the trajectories of the system \eqref{PSsyst.ext} and the stability of the point $Q_3$ give that the same property holds true in a right neighborhood $p\in(p_c(\sigma),p_c(\sigma)+\delta)$ for some $\delta>0$ sufficiently small, thus $p_0(\sigma)\geq p_c(\sigma)+\delta$ and the proof is complete.
\end{proof}
We plot in Figure \ref{fig3} the outcome of numerical experiments showing the difference between the two ranges $p<p_s(\sigma)$ and $p>p_s(\sigma)$, as described in the proof of Theorem \ref{th.extinction}.

\begin{figure}[ht!]
  \begin{center}
  \subfigure[$p<p_s(\sigma) \ closer \ to \ p_s(\sigma)$]{\includegraphics[width=7.5cm,height=6cm]{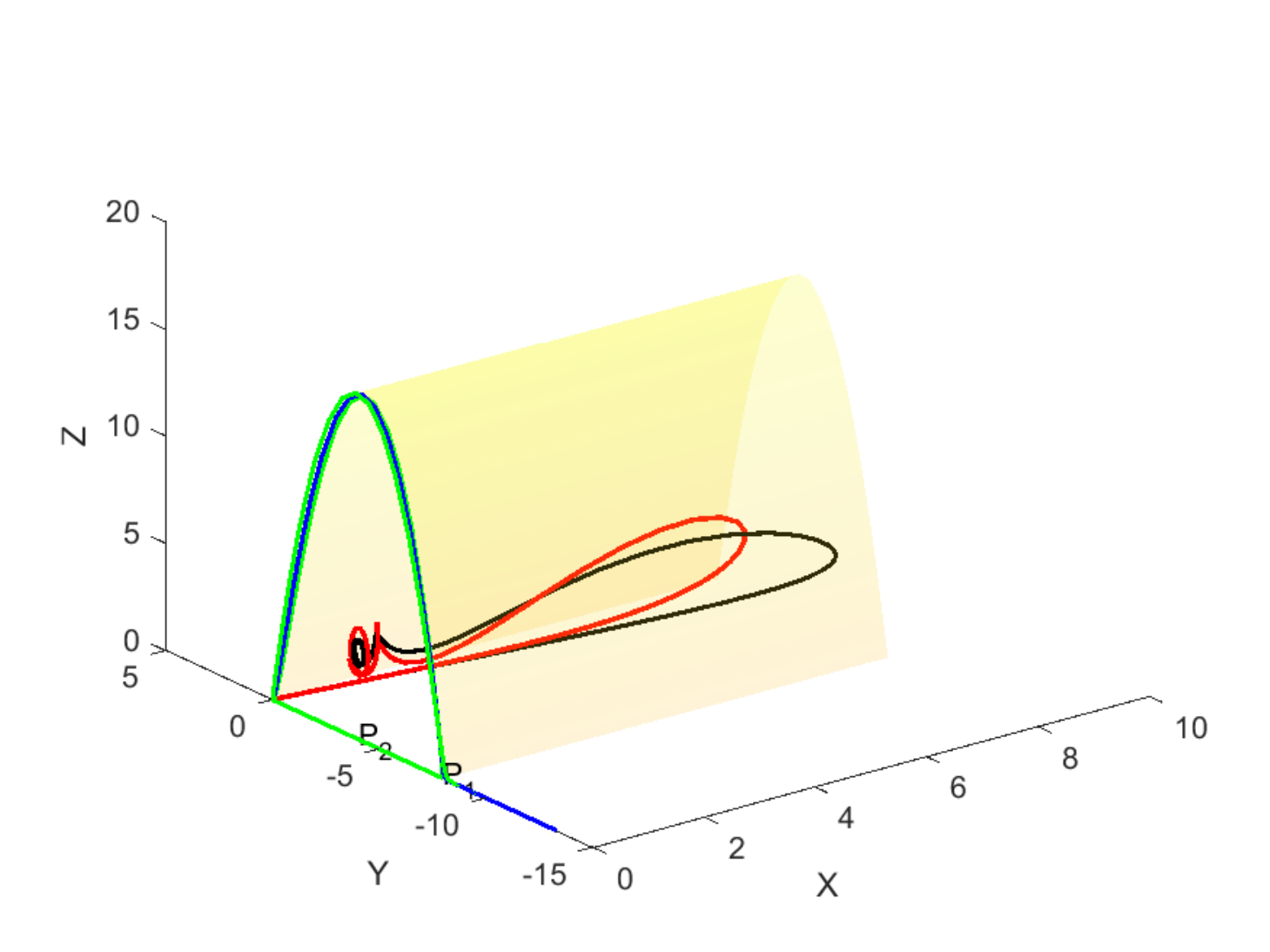}}
  \subfigure[$p>p_s(\sigma)$]{\includegraphics[width=7.5cm,height=6cm]{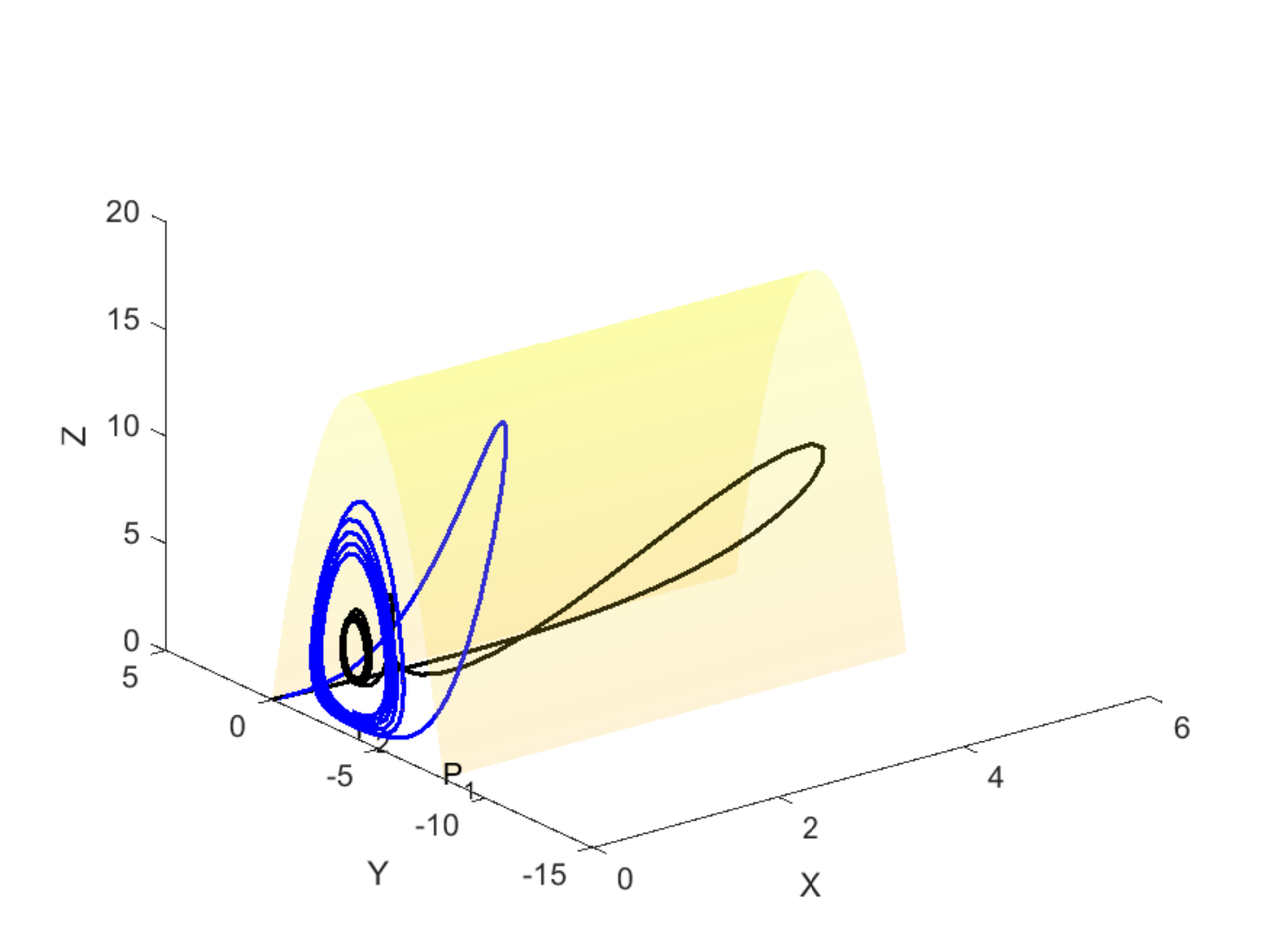}}
  \end{center}
  \caption{Orbits from $P_0$ for $p<p_s(\sigma)$ and for $p>p_s(\sigma)$ containing finite time extinction profiles. Experiments for $m=0.25$, $N=4$, $\sigma=4$ and $p=1.74$, respectively $p=1.8$, where $p_s(\sigma)=1.75$.}\label{fig3}
\end{figure}

\noindent \textbf{Remark}. Since $P_3$ lies in the interior of the cylinder \eqref{cylinder}, we notice that the unique orbit going out of $P_3$ connects to $Q_3$ for $p=p_c(\sigma)$ and remains to oscillate in the interior of the cylinder \eqref{cylinder} for $p=p_s(\sigma)$. A similar three-sets argument as the one performed with the sets $\mathcal{U}$, $\mathcal{V}$ and $\mathcal{W}$, but with the only difference that we now define these sets with respect to the orbit starting from $P_3$ for $p\in(p_c(\sigma),p_s(\sigma))$, ensures that there exists at least some $p_1(\sigma)\in(p_c(\sigma),p_s(\sigma))$ for which the orbit going out of $P_3$ enters the point $P_1$. In terms of profiles, we get at least one profile $f(\xi)$ with the vertical asymptote \eqref{beh.P3} as $\xi\to0$ and with the tail \eqref{beh.P1} as $\xi\to\infty$. The corresponding self-similar solution can be seen as a \emph{generalized solution} to Eq. \eqref{eq1} in $L^1$ sense, since it is integrable at both ends if $m\in(0,m_c)$. This solution has the form 
\begin{equation*}
\begin{split}
u(x,t)&=(T-t)^{\alpha}f(|x|(T-t)^{\beta})\sim C(T-t)^{\alpha-2\beta/(1-m)}|x|^{-2/(1-m)}\\&=C(T-t)^{1/(1-m)}|x|^{-2/(1-m)},
\end{split}
\end{equation*}
as $t\to T$, where we recall the values of $\alpha$ and $\beta$ given in \eqref{ss.exponents}. We thus obtain an interesting unbounded solution which vanishes in finite time and stays unbounded up to the extinction time.

\medskip 

\noindent \textbf{Uniqueness versus multiplicity}. A possible further development of this work is to analyze the uniqueness or non-uniqueness of the profiles with the fast tail behavior \eqref{beh.P1} as $\xi\to\infty$, that is, trajectories connecting $P_0$ and $P_1$ in the two systems \eqref{PSsyst} and \eqref{PSsyst.ext}. We expect (by numerical evidences) the forward self-similar solution with such property to be unique, while for the extinction profiles, we expect that for any positive integer $n$, there exists some $p_n(\sigma)\in(p_0(\sigma),p_s(\sigma))$, $p_n(\sigma)>p_{n-1}(\sigma)$, such that for any $p\in(p_n(\sigma),p_s(\sigma))$ there exist at least $n$ different orbits connecting $P_0$ and $P_1$ in the system \eqref{PSsyst.ext}, each of them with one more oscillation in terms of the function $Y(\eta)$ than the precedent one. This is intuitively suggested by the fact that for $p=p_s(\sigma)$, the orbits $l_K$ oscillate infinitely many times. We refrain from studying rigorously these questions here, and they are left for further works.

\section{A family of stationary solutions for $p=p_s(\sigma)$}\label{sec.sobolev}

We show in this short section how we arrived to the explicit stationary solutions \eqref{sol.sobolev}. In fact, among the solutions with the specific decay \eqref{beh.P1} as $\xi\to\infty$, the critical exponent $p_s(\sigma)$ separates between self-similar solutions with $\|u(t)\|_{\infty}$ decaying and self-similar solutions with $\|u(t)\|_{\infty}$ increasing with respect to time, hence we could get the intuition that exactly for $p=p_s(\sigma)$ there might exist solutions not depending on time. We thus look for radially symmetric stationary solutions to Eq. \eqref{eq1}, which solve
\begin{equation}\label{interm.sob1}
(u^m)_{rr}+\frac{N-1}{r}(u^m)_r+r^{\sigma}u^p=0, \qquad r=|x|.
\end{equation}
Performing the change of variable 
\begin{equation}\label{change.sob}
u(r)=r^{-(\sigma+2)/(p-m)}v(r)^{1/m},
\end{equation}
we find that $v(r)$ solves in general the following Euler-type differential equation 
$$
r^2v_{rr}+\left[\frac{(N-2)(p-p_s(\sigma))}{p-m}+1\right]rv_{r}-\frac{m(N-2)(\sigma+2)(p-p_c(\sigma))}{(p-m)^2}v+v^{p/m}=0,
$$
which strongly simplifies if we let $p=p_s(\sigma)$. More precisely, if we also let (as usual for Euler equations) $s=\ln\,r$, we notice that for $p=p_s(\sigma)$ the term involving the first order derivative vanishes and we obtain the following differential equation for $v(s)$:
\begin{equation}\label{interm.sob2}
v_{ss}-\frac{(N-2)^2}{4}v+v^{(N+2\sigma+2)/(N-2)}=0.
\end{equation}
By multiplying \eqref{interm.sob2} by $v_s$ and integrating once, we reduce \eqref{interm.sob2} to a first order equation 
\begin{equation}\label{interm.sob3}
v_s^2-\frac{(N-2)^2}{4}v^2+\frac{N-2}{N+\sigma}v^{2(N+\sigma)/(N-2)}=K.
\end{equation}
Since we are looking for solutions $u(r)$ with the local behavior \eqref{beh.P1} as $r\to\infty$, we readily infer from \eqref{change.sob} that we expect that $v(r)\to0$ as $r\to\infty$, and the same for $v(s)$ as $s=\ln\,r\to\infty$, whence we have to let $K=0$ in \eqref{interm.sob3}. In this case, \eqref{interm.sob3} can be integrated implicitly to obtain that 
$$
s\pm\frac{2\arctanh{\sqrt{1-4v(s)^{2(\sigma+2)/(N-2)}/(N-2)(N+\sigma)}}}{\sigma+2}=D,
$$
where $D\in\real$ is the integration constant. We further deduce after straightforward manipulations that 
$$
v(s)=\left\{\frac{(N-2)(N+\sigma)}{4}\left[1-\tanh^2\left(D-\frac{\sigma+2}{2}s\right)\right]\right\}^{(N-2)/2(\sigma+2)}.
$$
We next replace $s=\ln\,r$ and find after direct calculations using the definition of the $\tanh$ function and undoing the change of variable \eqref{change.sob} the following expression:
$$
u(r)=r^{-(N-2)/2m}\left[\frac{(N-2)(N+\sigma)}{Cr^{-(\sigma+2)}+C^{-1}r^{\sigma+2}+2}\right]^{(N-2)/2m(\sigma+2)},
$$
where $C=e^{2D}>0$, and the previous formula readily leads to the expression for $U_C(r)$ given in \eqref{sol.sobolev}.  

\section*{Appendix. Proof of Lemmas \ref{lem.Q4} and \ref{lem.Q4.ext}}
\setcounter{equation}{0}
\renewcommand{\theequation}{A.\arabic{equation}}

\begin{proof}[Proof of Lemma \ref{lem.Q4}]
We split the proof into three cases, according to whether the limits in Lemma \ref{lem.Q4} are taken as $\xi\to0$, as $\xi\to\xi_0\in(0,\infty)$ or as $\xi\to\infty$.

\medskip

\noindent \textbf{Case 1. Limit as $\xi\to0$}. Assume for contradiction that there exists a solution to Eq. \eqref{ODE.forward} such that \eqref{limits} holds true as $\xi\to0$. Since $p-m>0$, we deduce that $\lim\limits_{\xi\to0}f(\xi)=\infty$. We then show that there exists $\delta>0$ such that $f$ is decreasing on $(0,\delta)$. Indeed, if this is not true, then there exists a sequence $(\xi_n)_{n\geq1}$ of local minima of $f$ such that $\xi_n\to0$. Since $f'(\xi_n)=(f^m)'(\xi_n)=0$, we infer from evaluating \eqref{ODE.forward} at $\xi=\xi_n$ that
$$
f(\xi_n)\left(\xi_n^{\sigma}f(\xi_n)^{p-1}-\alpha\right)=-(f^m)''(\xi_n), \qquad {\rm for \ any} \ n\in\mathbf{N}.
$$
But the latter is a contradiction, since
$$
\lim\limits_{n\to\infty}f(\xi_n)=\lim\limits_{n\to\infty}(\xi_n^{\sigma}f(\xi_n)^{p-1}-\alpha)=+\infty,
$$
while $(f^m)''(\xi_n)\geq0$. Hence, $f$ is decreasing on some right neighborhood $(0,\delta)$ of the origin. We then let $\xi\in(0,\delta)$ and write \eqref{ODE.forward} in the following form
\begin{equation*}
(f^m)''(\xi)+f(\xi)\left(\frac{1}{2}\xi^{\sigma}f(\xi)^{p-1}-\alpha\right)-\beta\xi f'(\xi)+\frac{1}{2}\xi^{\sigma}f(\xi)^p+\frac{m(N-1)}{\xi}(f^{m-1}f')(\xi)=0.
\end{equation*}
Since $\xi\in(0,\delta)$, we infer that $-\beta\xi f'(\xi)>0$, while \eqref{limits} gives
\begin{equation}\label{interm5}
\lim\limits_{\xi\to0}f(\xi)\left(\frac{1}{2}\xi^{\sigma}f(\xi)^{p-1}-\alpha\right)=\infty
\end{equation}
and
\begin{equation}\label{interm6}
\frac{1}{2}\xi^{\sigma}f(\xi)^p+\frac{m(N-1)}{\xi}(f^{m-1}f')(\xi)=\xi^{\sigma}f(\xi)^p\left[\frac{1}{2}+\frac{m(N-1)f'(\xi)}{\xi^{\sigma+1}f(\xi)^{p-m+1}}\right]>0,
\end{equation}
since the limit as $\xi\to0$ of the term in brackets in \eqref{interm6} is one half. Gathering \eqref{interm6}, \eqref{interm5} and the previous considerations, we find that
\begin{equation}\label{interm7}
\lim\limits_{\xi\to0}(f^m)''(\xi)=-\infty.
\end{equation}
In particular, \eqref{interm7} implies that there exists $\xi_1>0$ such that $(f^m)''(\xi)<-1$ if $\xi\in(0,\xi_1)$. By integrating on $(\xi,\xi_1)$ for a generic $\xi\in(0,\xi_1/2)$, we get
$$
(f^m)'(\xi_1)-(f^m)'(\xi)<-(\xi_1-\xi)=\xi-\xi_1,
$$
or equivalently
\begin{equation}\label{interm8}
(f^m)'(\xi)>(f^m)'(\xi_1)+\xi_1-\xi>(f^m)'(\xi_1)+\frac{\xi_1}{2}, \qquad \xi\in\left(0,\frac{\xi_1}{2}\right).
\end{equation}
We integrate again over $(\xi,\xi_1)$ in \eqref{interm8} and obtain after obvious manipulations
$$
f^m(\xi)<f^m(\xi_1)+(\xi-\xi_1)\left[\frac{\xi_1}{2}+(f^m)'(\xi_1)\right], \qquad \xi\in\left(0,\frac{\xi_1}{2}\right),
$$
which leads to a contradiction by taking limits as $\xi\to0$ and recalling that $f(\xi)\to\infty$ as $\xi\to0$.

\medskip

\noindent \textbf{Case 2. Limit as $\xi\to\xi_0\in(0,\infty)$}. Assume for contradiction that there exists a solution to \eqref{ODE.forward} such that \eqref{limits} holds true as $\xi\to\xi_0\in(0,\infty)$. We deduce that $f(\xi)\to\infty$ as $\xi\to\xi_0$ and exactly as in the previous step, we obtain that $f$ is either decreasing on some right neighborhood $(\xi_0,\xi_0+\delta)$ if the limit as $\xi\to\xi_0$ is taken from the right, or increasing on some left neighborhood $(\xi_0-\delta,\xi_0)$ if the limit as $\xi\to\xi_0$ is taken from the left. In the former case, the analysis performed in Case 1 works similarly and leads to a contradiction. We are left with the latter case, that is, $f$ increasing on $(\xi_0-\delta,\xi_0)$ for some $\delta>0$ and with a vertical asymptote from the left as $\xi\to\xi_0$, $\xi<\xi_0$. Since $\xi^{\sigma-2}f(\xi)^{m+p-2}\to\infty$ as $\xi\to\xi_0$, we derive that this case is only possible if $m+p>2$. Moreover, \eqref{limits} gives
$$
\lim\limits_{\xi\to\xi_0}f(\xi)\left(\frac{1}{2}\xi^{\sigma}f(\xi)^{p-1}-\alpha\right)=\infty, \qquad \frac{N-1}{\xi}(f^m)'(\xi)>0,
$$
the latter holding true for $\xi\in(\xi_0-\delta,\xi_0)$, thus we infer from \eqref{ODE.forward} and the positivity of $f(\xi)$ that, on the one hand,
\begin{equation}\label{interm9}
\lim\limits_{\xi\to\xi_0}\left[(f^m)''(\xi)-\beta\xi f'(\xi)\right]=-\infty,
\end{equation}
while on the other hand, there exists a sequence $(\xi_n)_{n\geq1}$ such that $\xi_n\in(\xi_0-\delta,\xi_0)$ for any natural number $n$, $\xi_n\to\xi_0$ and
\begin{equation}\label{interm10}
f'(\xi_n)>\frac{1}{2\beta}\xi_n^{\sigma-1}f(\xi_n)^p, \qquad {\rm for \ any} \ n\in\mathbf{N}.
\end{equation}
The first condition implies that, if we fix some $K>0$ sufficiently large, there is $\xi_0(K)\in(\xi_0-\delta,\xi_0)$ such that
$$
(f^m)''(\xi)-\beta\xi_0f'(\xi)<-K, \qquad {\rm for} \ \xi\in(\xi_0(K),\xi_0),
$$
which after an integration over $(\xi_0(K),\xi)$ leads to
\begin{equation*}
\begin{split}
(f^m)'(\xi)-(f^m)'(\xi_0(K))&<\beta\xi_0(f(\xi)-f(\xi_0(K)))-K(\xi-\xi_0(K))\\
&<\beta\xi_0(f(\xi)-f(\xi_0(K)))-\frac{K}{2}(\xi_0-\xi_0(K)),
\end{split}
\end{equation*}
provided that we choose $\xi\in((\xi_0+\xi_0(K))/2,\xi_0)$. The latter estimate can be written in an equivalent form as follows
\begin{equation}\label{interm11}
mf(\xi)^{m-1}f'(\xi)<\beta\xi_0 f(\xi)+(f^m)'(\xi_0(K))-\frac{K}{2}(\xi_0-\xi_0(K))-\beta\xi_0 f(\xi_0(K)).
\end{equation}
Noticing that the condition $m+p>2$ together with the limitation $p<p_L(\sigma)$, gives
$$
\sigma>\frac{2(p-1)}{1-m}>2,
$$
the condition \eqref{interm10} implies, if $\xi_n\in(\xi_0(K),\xi_0)$, that
\begin{equation}\label{interm12}
f'(\xi_n)>\frac{1}{2\beta}\xi_n^{\sigma-1}f(\xi_n)^p>\frac{\xi_0(K)^{\sigma-1}}{2\beta}f(\xi_n)^p.
\end{equation}
We evaluate \eqref{interm11} at $\xi=\xi_n$ for $n$ sufficiently large such that $\xi_n>(\xi_0+\xi_0(K))/2$ and replace $f'(\xi_n)$ by the estimate \eqref{interm12} to obtain
$$
\frac{m\xi_0(K)^{\sigma-1}}{2\beta}f(\xi_n)^{m+p-1}<\beta\xi_0 f(\xi_n)+(f^m)'(\xi_0(K))-\beta\xi_0 f(\xi_0(K))-\frac{K}{2}(\xi_0-\xi_0(K)),
$$
or equivalently
\begin{equation}\label{interm13}
f(\xi_n)^{m+p-2}<\frac{2\beta^2\xi_0}{m\xi_0(K)^{\sigma-1}}+\frac{C(K,\xi_0)}{f(\xi_n)},
\end{equation}
where
$$
C(K,\xi_0):=\frac{2\beta}{m\xi_0(K)^{\sigma-1}}\left[(f^m)'(\xi_0(K))-\beta\xi_0 f(\xi_0(K))-\frac{K}{2}(\xi_0-\xi_0(K))\right].
$$
By passing to the limit as $n\to\infty$ in \eqref{interm13} and recalling that $f(\xi_n)\to\infty$ (since $\xi_n\to\xi_0$) and $m+p-2>0$ we reach an obvious contradiction.

\medskip

\noindent \textbf{Case 3. Limit as $\xi\to\infty$}. Assume for contradiction that there exists a solution to \eqref{ODE.forward} such that \eqref{limits} holds true as $\xi\to\infty$. We can prove in the same way as in Case 1 that there exists $\xi_0\in(0,\infty)$ such that $f$ is monotone on $(\xi_0,\infty)$ and in particular, that there exists $\lim\limits_{\xi\to\infty}f(\xi)$. We divide the analysis into three subcases according to the value of this limit.

$\bullet$ \textbf{If} $\mathbf{\lim\limits_{\xi\to\infty}f(\xi)=0}$, we infer that $f$ is decreasing on $(\xi_0,\infty)$. In this case, we observe that $-\beta\xi f'(\xi)>0$ for $\xi\in(\xi_0,\infty)$ and
\begin{equation*}
\begin{split}
\frac{N-1}{\xi}(f^m)'(\xi)-\alpha f(\xi)+\xi^{\sigma}f(\xi)^p&=\xi^{\sigma}f(\xi)^p\left[\frac{1}{2}+\frac{m(N-1)f'(\xi)}{\xi^{\sigma+1}f(\xi)^{p-m+1}}\right]\\
&+f(\xi)\left[\frac{1}{2}\xi^{\sigma}f(\xi)^{p-1}-\alpha\right]>0,
\end{split}
\end{equation*}
since both terms in brackets are positive for $\xi$ sufficiently large according to \eqref{limits}. It follows that $(f^m)''(\xi)<0$ for any $\xi>\xi_0$ sufficiently large, hence $(f^m)'(\xi)<0$ and decreasing in the same interval, which leads to an obvious contradiction to the horizontal asymptote of $f^m$.

$\bullet$ \textbf{If} $\mathbf{\lim\limits_{\xi\to\infty}f(\xi)=l\in(0,\infty)}$, then either $f$ is decreasing on some interval $(\xi_0,\infty)$, and in this case the contradiction can be obtained exactly in the same way as in the previous paragraph when the limit was zero, or $f$ is increasing on $(\xi_0,\infty)$. In this latter case, we first apply \cite[Lemma 2.9]{IL13} to the function $l-f$ to infer that there exists a subsequence $(\xi_n)_{n\geq1}$ such that $\xi_n f'(\xi_n)\to0$ as $n\to\infty$. Since $f'\geq0$ on $(\xi_0,\infty)$, we can apply once more \cite[Lemma 2.9]{IL13} to the function $\xi f'(\xi)$ and conclude that there exists a subsequence (relabeled $(\xi_n)$) such that $\xi_n f'(\xi_n)\to0$ and $\xi_n^2f''(\xi_n)\to0$ as $n\to\infty$, and the latter implies obviously that also $f'(\xi_n)\to0$ and $f''(\xi_n)\to0$. A simple calculation and the fact that $f(\xi_n)\to l\in(0,\infty)$ further gives that $(f^m)''(\xi_n)\to0$ as $n\to\infty$. By evaluating \eqref{ODE.forward} at $\xi=\xi_n$, taking into account the previous limits and the fact that
$$
\lim\limits_{n\to\infty}f(\xi_n)(\xi_n^{\sigma}f(\xi_n)^{p-1}-\alpha)=\infty,
$$
we obtain a contradiction since the previous infinite limit cannot be compensated by any other term in the equation \eqref{ODE.forward}.

$\bullet$ \textbf{If} $\mathbf{\lim\limits_{\xi\to\infty}f(\xi)=\infty}$, it also means that $f$ is increasing on some interval $(\xi_0,\infty)$. We then decompose the equation \eqref{ODE.forward} in the following way
\begin{equation*}
(f^m)''(\xi)+\frac{1}{3}\xi^{\sigma}f(\xi)^{p}+\frac{1}{3}\xi^{\sigma}f(\xi)^{p}-\beta\xi f'(\xi)+\frac{1}{3}\xi^{\sigma}f(\xi)^{p}-\alpha f(\xi)+\frac{N-1}{\xi}(f^m)'(\xi)=0.
\end{equation*}
On the one hand, by a direct calculation, we obtain
\begin{equation}\label{interm15}
(f^m)''(\xi)+\frac{1}{3}\xi^{\sigma}f(\xi)^{p}=\xi^{\sigma}f(\xi)^p\left[\frac{1}{3}+\frac{mf'(\xi)}{\xi^{\sigma+1}f(\xi)^{p-m+1}}\frac{\xi f''(\xi)}{f'(\xi)}\right]+m(f^{m-1})'(\xi)f'(\xi).
\end{equation}
On the other hand, \eqref{limits} and the monotonicity of $f$ ensure that
\begin{equation}\label{interm16}
\lim\limits_{\xi\to\infty}f(\xi)\left(\frac{1}{3}\xi^{\sigma}f(\xi)^{p-1}-\alpha\right)=\infty, \qquad \frac{N-1}{\xi}(f^m)'(\xi)>0.
\end{equation}
We thus infer from \eqref{ODE.forward}, \eqref{interm15} and \eqref{interm16} that
\begin{equation}\label{interm17}
\begin{split}
\lim\limits_{\xi\to\infty}&\left\{\xi^{\sigma}f(\xi)^p\left[\frac{1}{3}+\frac{mf'(\xi)}{\xi^{\sigma+1}f(\xi)^{p-m+1}}\frac{\xi f''(\xi)}{f'(\xi)}\right]+\frac{1}{3}\xi^{\sigma}f(\xi)^{p}\right.\\&\left.-f'(\xi)\left[\beta\xi-m(f^{m-1})'(\xi)\right]\right\}=-\infty.
\end{split}
\end{equation}
Assume now for contradiction that there exists $\delta>1$ and $\xi_0(\delta)>\xi_0$ such that
$$
\frac{\xi f''(\xi)}{f'(\xi)}<-\delta, \qquad {\rm for} \ \xi>\xi_0(\delta).
$$
Since $f'(\xi)>0$ for $\xi>\xi_0(\delta)$, we get that $(\xi^{\delta}f')'(\xi)<0$ for $\xi>\xi_0(\delta)$, hence the function $\xi\mapsto\xi^{\delta}f'(\xi)$ is decreasing and in particular,
$$
f'(\xi)<\xi_0(\delta)^{\delta}f'(\xi_0(\delta))\xi^{-\delta}, \qquad {\rm for} \ \xi>\xi_0(\delta),
$$
which gives by integration over the interval $(\xi_0(\delta),\xi)$ that
\begin{equation}\label{interm18}
f(\xi)<f(\xi_0(\delta))+\xi_0(\delta)^{\delta}f'(\xi_0(\delta))\frac{\xi_0(\delta)^{1-\delta}-\xi^{1-\delta}}{\delta-1},
\end{equation}
which readily leads to a contradiction by passing to the limit as $\xi\to\infty$ and recalling that $\delta>1$ and $f(\xi)\to\infty$. We thus infer from these considerations and the fourth limit in \eqref{limits} that the first big term in \eqref{interm17} cannot tend to $-\infty$, hence it should be the second part of \eqref{interm17} the one tending to $-\infty$, thus
\begin{equation}\label{interm19}
\lim\limits_{\xi\to\infty}\left[\frac{1}{3}\xi^{\sigma}f(\xi)^{p}-\beta\xi f'(\xi)+mf'(\xi)(f^{m-1})'(\xi)\right]=-\infty.
\end{equation}
This implies in particular that there exists $\xi_1>0$ such that
$$
\frac{1}{3}\xi^{\sigma}f(\xi)^{p}<\beta\xi f'(\xi), \qquad {\rm for} \ \xi>\xi_1.
$$
We obtain after one integration on $(\xi_1,\xi)$ that
\begin{equation}\label{interm20}
\frac{\xi^{\sigma}-\xi_1^{\sigma}}{3\sigma\beta}<\frac{1}{1-p}(f(\xi)^{1-p}-f(\xi_1)^{1-p}),
\end{equation}
and we reach a contradiction by passing to the limit as $\xi\to\infty$ in \eqref{interm20} and recalling that $f(\xi)\to\infty$ and $p>1$. The proof is now complete.
\end{proof}
We now pass to the proof of Lemma \ref{lem.Q4.ext}, where we only give a sketch, as it goes along the same lines.
\begin{proof}[Sketch of the proof of Lemma \ref{lem.Q4.ext}]
According to the statement, we only have to consider limits as $\xi\to\xi_0$, $\xi<\xi_0$, for some $\xi_0\in(0,\infty)$ or as $\xi\to\infty$. In the former case, \eqref{limits} implies that $f(\xi)\to\infty$ as $\xi\to\xi_0$, $\xi<\xi_0$. A similar argument with points of minima as in Case I of the proof of Lemma \ref{lem.Q4} implies that $f(\xi)$ is increasing on an interval $(\xi_0-\delta,\xi_0)$, and as all the other terms are positive, this fact and Eq. \eqref{ODE.extinction} imply that $(f^m)''(\xi)\to-\infty$ as $\xi\to\xi_0$, $\xi<\xi_0$, and we proceed as in Case I in the proof of Lemma \ref{lem.Q4} to reach a contradiction. 

If we assume that there is a solution $f(\xi)$ to Eq. \eqref{ODE.extinction} satisfying \eqref{limits} as $\xi\to\infty$, in the same way we show that $f(\xi)$ is monotone in some interval $\xi\in(\xi_1,\infty)$ and thus there exists $l=\lim\limits_{\xi\to\infty}f(\xi)$. Again, if we assume that $f(\xi)$ is increasing we reach a similar contradiction as in the previous paragraph by getting that $\lim\limits_{\xi\to\infty}(f^m)''(\xi)=-\infty$. It follows that $f(\xi)$ is decreasing on $(\xi_1,\infty)$ and $l<\infty$. In this case, if $l>0$, an argument based on applying \cite[Lemma 2.9]{IL13} to the function $\xi\mapsto f(\xi)-l$, together with the fact that 
$$
\lim\limits_{\xi\to\infty}(\xi^{\sigma}f(\xi)^{p-1}+\alpha)=\infty,
$$ 
leads to a contradiction exactly in the same way as in the similar step in Case 3 in the proof of Lemma \ref{lem.Q4}. We are left with the case $l=0$. In this case, we write the equation \eqref{ODE.extinction} in the form
\begin{equation*}
-(f^m)''(\xi)=f(\xi)\left[\xi^{\sigma}f(\xi)^{p-1}\left(\frac{1}{2}+\frac{m(N-1)f'(\xi)}{\xi^{\sigma+1}f(\xi)^{p-m+1}}\right)+\alpha+\frac{1}{2}\xi^{\sigma}f(\xi)^{p-1}+\frac{\beta\xi f'(\xi)}{f(\xi)}\right],
\end{equation*}
and we readily find that the right hand side of the previous equality is positive for $\xi\in(\xi_1,\infty)$ since the fact that $\xi f'(\xi)/f(\xi)\to-\infty$ is contradictory with the fact that $\xi^{\sigma}f(\xi)^{p-1}\to\infty$ as $\xi\to\infty$. We thus infer that $(f^m)''(\xi)<0$ for $\xi>0$ sufficiently large and again a contradiction with the horizontal asymptote of $f^m$.
\end{proof}

\bigskip

\noindent \textbf{Acknowledgements} R. I. and A. S. are partially supported by the Spanish project PID2020-115273GB-I00.

\bibliographystyle{plain}

\end{document}